\documentclass[11pt,reqno]{article}

\usepackage{amsmath,amsfonts,amssymb,graphics,amsthm,comment,enumerate}
\usepackage{hyperref}
\hypersetup{
    colorlinks=true,
    linkcolor=blue,
    citecolor=red,
    urlcolor=blue,
    pdfborder={0 0 0}
}    
\usepackage{lipsum}
\usepackage{graphicx}
\usepackage{xcolor}
\usepackage{bbm}
\usepackage{wrapfig} 
\usepackage[font=sf, labelfont={sf,bf}, margin=1cm]{caption}
\usepackage{geometry}
\newgeometry{vmargin={30 mm}, hmargin={18mm,18mm}}

\usepackage{cleveref}
  \crefname{theorem}{Theorem}{Theorems}
  \crefname{thm}{Theorem}{Theorems}
  \crefname{lemma}{Lemma}{Lemmas}
  \crefname{lem}{Lemma}{Lemmas}
  \crefname{remark}{Remark}{Remarks}
  \crefname{prop}{Proposition}{Propositions}
  \crefname{proposition}{Proposition}{Propositions}
  \crefname{question}{Question}{Questions}
\crefname{notation}{Notation}{Notations}
\crefname{claim}{Claim}{Claims}
  \crefname{defn}{Definition}{Definitions}
  \crefname{corollary}{Corollary}{Corollaries}
  \crefname{section}{Section}{Sections}
  \crefname{figure}{Figure}{Figures}
  \crefname{exercise}{Exercise}{Exercises}
    \crefname{assumption}{Assumption}{Assumptions}

\newtheorem{thm}{Theorem}[section]

\newtheorem{lemma}[thm]{Lemma}
\newtheorem{corollary}[thm]{Corollary}
\newtheorem{prop}[thm]{Proposition}

\newtheorem{question}[thm]{Question}

\numberwithin{equation}{section}
\theoremstyle{definition}
\newtheorem{remark}[thm]{Remark}

\def\cX{\mathcal{X}}

\def\cT{\mathcal{T}}

\def\cP{\mathcal{P}}

\def\cE{\mathcal{E}}

\def\cC{\mathcal{C}}

\def\cA{\mathcal{A}}

\def\P{\mathbb{P}}
\def\E{\mathbb{E}}

\def\Z{\mathbb{Z}}

\def\V{\mathbb{V}}
\def\L{\mathbb{L}}

\newcommand{\1}{\mathbf{1}}

\def \ve {\varepsilon}

\DeclareMathOperator{\dist}{dist}

\DeclareMathOperator{\bP}{\mathsf P}
\DeclareMathOperator{\bZ}{\mathsf Z}

\newcommand{\zero}{\mathbf{0}}
\newcommand{\iid}{i.i.d.}
\newcommand{\ffiid}{ffiid}

\DeclareMathOperator{\diam}{diam}

\title{Finitary codings for gradient models and a new graphical representation for the six-vertex model}
\date{\today}

\author{
	Gourab Ray
	\thanks{University of Victoria, Department of Mathematics, Victoria, BC, V8W 2Y2, Canada. Supported in part by NSERC 50311-57400 and University of Victoria start-up 10000-27458.}
	\and
	Yinon Spinka
	\thanks{University of British Columbia,
    Department of Mathematics,
 	Vancouver, BC, V6T 1Z2, Canada. Supported in part by NSERC of Canada.}
}

\AtEndDocument{
  \bigskip
  \small
  \textit{Emails:} \texttt{gourabray@uvic.ca, yinon@math.ubc.ca} \par
}

\begin{document}
\maketitle

\begin{abstract}
It is known that the Ising model on $\Z^d$ at a given temperature is a finitary factor of an \iid\ process if and only if the temperature is at least the critical temperature. Below the critical temperature, the plus and minus states of the Ising model are distinct and differ from one another by a global flip of the spins. We show that it is only this global information which poses an obstruction to being finitary by showing that the gradient of the Ising model is a finitary factor of \iid\ at all temperatures. As a consequence, we deduce a volume-order large deviation estimate for the energy. Results in the same spirit are shown for the Potts model, the so-called beach model, and the six-vertex model. We also introduce a coupling between the six-vertex model with $c\ge 2$ and a new Edwards--Sokal type graphical representation of it, which we believe is of independent interest.
\end{abstract}

\section{Introduction}
\label{sec:intro}

A \emph{factor of an \iid\ process} on $\Z^d$ is any random field of the form $\varphi(Y)$, where $Y=(Y_v)_{v \in \Z^d}$ is an \iid\ process and $\varphi$ is a measurable function which commutes with translations of $\Z^d$.
Such a factor is \emph{finitary} if in order to compute the value at the origin, one only needs to observe a finite (but random) portion of the \iid\ process, i.e., if there almost surely exists a finite $R$ such that $(Y_v)_{|v| \le R}$ determines the value of $\varphi(Y)_\zero$. In such a case, we say that the random field is a finitary factor of an \iid\ process and we abbreviate this as \emph{\ffiid}.

Ornstein and Weiss~\cite{OW73} (see~\cite{adams1992folner} for a published version) showed that the plus state of the Ising model on~$\Z^d$ at any positive temperature is a factor of an \iid\ process (which is the same as the ergodic-theoretical notion of Bernoulli), thus indicating that the notion of a (non-finitary) factor of \iid\ is not sensitive enough to detect a phase transition in models of statistical mechanics.
In contrast, van den Berg and Steif~\cite{van1999existence} showed that the plus state of the Ising model at a given temperature is \ffiid\ if and only if there is a unique Gibbs measure at that temperature (which is the case if and only if the temperature is at least the critical temperature), so that the notion of a finitary factor of \iid\ aligns precisely with that of a phase transition in this case.
An analogous statement is also known to hold for the Potts model, as well as for monotonic (FKG) models under mild technical assumptions~\cite{spinka2018finitarymrf,harel2018finitary}.
Though such a characterization has not been established for general spin systems, there are additional examples of models where a similar picture emerges~(see~\cite{spinka2018finitarymrf}).
These results point to a close connection between the existence or non-existence of a finitary factor of \iid\ and the more classical phenomenon of phase transition in spin systems.

The results showing the non-existence of a finitary factor of i.i.d.\ when multiple Gibbs states exist are not very illuminating as to the true nature of the obstruction. One goal of this work is to shed more light on the reason as to why such models are not ffiid when multiple Gibbs states exist. In many situations, the extreme invariant Gibbs states are related to one another by a simple transformation involving a global permutation of the spins. Morally, it is this global information which cannot be coded in a finitary manner.
The common narrative behind the results shown here is that once this information is discarded (in a suitable manner), the remaining information can be coded in a finitary manner. We informally call the remaining information a \emph{gradient} of the original model. As we shall see, this nomenclature is natural for the models considered here as this information can be interpreted as a discrete gradient.

We present results regarding the above phenomenon in three examples of statistical mechanics -- the Ising and Potts models, the beach model, and the six-vertex model. We state here brief and informal versions of our results for each model. The precise versions and relevant definitions appear in the later sections.

Our first result concerns the Ising and Potts models at any temperature below the critical temperature. As mentioned, it is known in this case that no Gibbs measure is \ffiid. We define a natural gradient of these models: for the Ising model, this is the edge percolation consisting of those edges whose endpoints have different signs, and for the Potts model, we consider the difference along edges modulo $q$. See \cref{sec:Ising} for detailed definitions and results.

\begin{thm}\label{thm:pre_ising_potts}
Fix $d,q \ge 2$.
\begin{itemize}
 \item The gradient of the plus state of the Ising on $\Z^d$ at inverse temperature~$\beta>\beta_c(d)$ is \ffiid\ with a coding radius having exponential tails.
 \item The gradient of any constant-boundary Gibbs measure for the $q$-state Potts model on $\Z^d$ at inverse temperature $\beta>0$ is \ffiid\ if and only if the free and wired measures of the associated random-cluster model coincide.
\end{itemize}
\end{thm}

We mention a curious observation here. Suppose $X$ is sampled from the plus state of the Ising model and let $Y$ be its gradient.
In the uniqueness regime ($\beta \le \beta_c$), where $X$ (and hence also $Y$) is \ffiid, the gradient mods out a single global bit (in ergodic-theoretical terms, $X$ is a two-point extension of $Y$), and so one cannot recover $X$ as an almost sure function of $Y$.
At low temperatures ($\beta>\beta_c$) on the other hand, the gradient operation is non-lossy (it does not lose information) and $X$ can be recovered from $Y$ (given the gradient, there is a unique choice of spins which results in a higher density of pluses). This is curious since $X$ is not \ffiid, while $Y$, which is obtained from $X$ via a non-lossy continuous map, is \ffiid.

Our second result concerns the beach model of Burton and Steif~\cite{burton1994non}. The model bears a strong analogy with the Potts model, with spins having a $\{0,1\}$-valued state in addition to a $q$-valued type. Above a certain critical parameter, the model admits $q$ ordered Gibbs measures, all of which are not \ffiid. We consider a gradient of this model, which preserves the $\{0,1\}$-valued state and is applied to the types in the same manner as for the Potts model, and show that this gradient is \ffiid\ (see \cref{sec:beach-model} for details).

\begin{thm}\label{thm:pre_beach}
Fix $d,q \ge 2$.
The gradient of any constant-type Gibbs measure for the $q$-type beach model on $\Z^d$ at fugacity~$\lambda>0$ is \ffiid\ if and only if the free and wired measures of the associated beach-random-cluster model coincide.
\end{thm}

Our third result concerns the six-vertex model (more precisely, a specific version of the six-vertex model called the F-model) with large parameter $c$, where we consider the Gibbs measures arising from ``flat'' boundary conditions. The model admits a natural integer-valued height function representation. We consider the discrete gradient and ``Laplacian'' of this height function, and show that while they are not \ffiid, the absolute value of the latter is. To this end, we introduce and initiate the study of a new graphical representation of the six-vertex model with parameter $c \ge 2$.
We believe that this representation is also a major contribution of this paper and is of independent interest (see \cref{sec:6v} for details).

\begin{thm}\label{thm:pre_six_vertex}
For the six-vertex model with $c$ large enough:
\begin{itemize}
 \item The height function, its gradient and its Laplacian are not \ffiid.
 \item The absolute value of the Laplacian is \ffiid.
\end{itemize}
\end{thm}

The existence of a finitary coding in all of the above results relies on a more general result, given in \cref{sec:general}, which roughly says that if a model has a random-cluster-type representation which is known to be \ffiid\ and has a unique infinite cluster, then the gradient of the model is also \ffiid.

\subsection{Outline of proof}

We focus our attention here on an outline of the proof that the gradient of the Potts model is \ffiid\ when the associated free and wired random-cluster measures coincide (the ``if'' part of the second item in \cref{thm:pre_ising_potts}). 

An indispensable tool in the study of this model is the random-cluster model, which serves as a graphical representation of the Potts model. We do not recall the definition of this model here, but only that it is an edge percolation model and that it is related to the Potts model via the Edwards--Sokal coupling which can be described as follows. Let $\sigma$ have the law of a constant 0 boundary condition Gibbs state for the Potts model and let $\omega$ have the law of the associated wired random-cluster measure.
The two are coupled in such a way that if an edge is present in $\omega$, then the spins at its endpoints are forced to be equal in $\sigma$.
Subject to this constraint, the coupling is essentially as simple as possible: in one direction, given the spin configuration $\sigma$, the random-cluster configuration $\omega$ is obtained via an independent edge percolation on clusters of constant spin with a parameter depending on the temperature. In the other direction, given the random-cluster configuration, the spins are obtained by independently assigning a uniform spin to each finite cluster, and assigning spin 0 to the infinite cluster.

A recent result from \cite{harel2018finitary} shows that $\omega$ is \ffiid\ precisely when the free and wired random-cluster measures coincide. Thus, it is enough to show that when $\sigma$ is obtained from $\omega$ as above, its gradient is a finitary factor of $\omega$ and an additional independent \iid\ process.
This is not immediate, as $\omega$ contains an infinite cluster and it is not possible to figure out whether a vertex is in the infinite cluster in a finitary manner (that is, the assignment of spin 0 to the infinite cluster requires looking at infinitely many edges in $\omega$). We get around this problem by constructing a rooted tree structure on the clusters of $\omega$ in which the infinite cluster is the root and with the property that the tree can (in a certain sense) be obtained from $\omega$ in a finitary manner (this part of the argument is general and works for any percolation process with a unique infinite cluster; see \cref{sec:general}). Assigning independent spins to the \emph{finite} clusters, this tree structure allows us to view these spins, not as the actual value of the spins in the cluster (which would be the straightforward way to implement the Edwards--Sokal coupling), but rather as a difference (mod $q$) between the value of the spins in that cluster and its parent cluster. The gradient along a directed edge $(u,v)$ may then be computed by traveling along the tree, first up the tree from $u$ to its lowest common ancestor with $v$, and then back down the tree to $v$, adding the spins along the way (with a negative sign when going down), disregarding the spin of the lowest common ancestor (which might be the infinite cluster).
This will show that the gradient of $\sigma$ is a finitary factor of $\omega$ and an independent \iid\ process, and will hence allow us to conclude that the gradient is \ffiid.

\smallskip

The ``only if'' part of the second item in \cref{thm:pre_ising_potts} will follow easily from the results in~\cite{harel2018finitary} and the Edwards--Sokal coupling.

\smallskip

For the first item in \cref{thm:pre_ising_potts}, we additionally use the known fact that the free and wired FK-Ising measures (the random-cluster measure with cluster weight $q=2$) coincide for all values of the parameter $p$. Combining this with Pisztora's coarse graining approach gives us good control on the coding radius as well. 

\smallskip

For the beach model (\cref{thm:pre_beach}), a similar approach works using an analogous random-cluster representation introduced by Haggstr\"om \cite{haggstrom1996random,haggstrom1998random}.
For this, we use the fact that this representation has monotonicity properties (FKG) and a general result from~\cite{harel2018finitary} about finitary codings for monotone models with unique Gibbs measures.

\smallskip

For the six-vertex model (\cref{thm:pre_six_vertex}), we construct a similar representation and prove the necessary monotonicity properties and uniqueness of Gibbs measure for large $c$. This allows us to apply our general result. A short time after the first draft of this paper appeared online, it came to our attention that Marcin Lis \cite{lis2019spins} had independently obtained a similar representation for a wider class of models.

\subsection{Coding definitions}
While our main results only deal with models on $\Z^d$, some of the relevant random fields which arise are defined on the vertices of $\Z^d$ and some on the edges of $\Z^d$. Also, we will be concerned with gradients of these models, which naturally live on slightly modified graphs. For these reasons, we give the definitions below for general graphs and not just for $\Z^d$.

Let $G$ be a transitive locally-finite graph on a countable vertex set $\V$, and let $\Gamma$ be a transitive subgroup of the automorphism group of $G$. A \emph{random field (or random process)} on $G$ is a collection of random variables $X=(X_v)_{v \in \V}$ indexed by the vertices of $G$ and defined on a common probability space.
We say that $X$ is \emph{$\Gamma$-invariant} if its distribution is not affected by the action of $\Gamma$, i.e., if $(X_{\gamma v})_{v \in \V}$ has the same distribution as $X$ for any $\gamma \in \Gamma$.

Let $S$ and $T$ be two measurable spaces, and let $X=(X_v)_{v \in \V}$ and $Y=(Y_v)_{v \in \V}$ be $S$-valued and $T$-valued $\Gamma$-invariant random fields.
A \emph{coding} from $Y$ to $X$ is a measurable function $\varphi \colon T^\V \to S^\V$, which is \emph{$\Gamma$-equivariant}, i.e., commutes with the action of every element of $\Gamma$, and which satisfies that $\varphi(Y)$ and $X$ are identical in distribution. Such a coding is also called a \emph{factor map} from $Y$ to~$X$, and when such a coding exists, we say that $X$ is a \emph{$\Gamma$-factor} of~$Y$.

Suppose now that $S$ and $T$ are countable.
Let $\zero \in \V$ be a distinguished vertex.
The \emph{coding radius} of $\varphi$ at a point $y \in T^\V$, denoted by $R(y)$, is the minimal integer $r \ge 0$ such that $\varphi(y')_\zero=\varphi(y)_\zero$ for all $y' \in T^\V$ which coincide with $y$ on the ball of radius $r$ around $\zero$ in the graph-distance, i.e., $y'_v=y_v$ for all $v \in \V$ such that $\dist(v,\zero) \le r$. It may happen that no such~$r$ exists, in which case, $R(y)=\infty$.
Thus, associated to a coding is a random variable $R=R(Y)$ which describes the coding radius.
While $S$ will always be finite or countable, we will allow $T$ to be a larger space, in which case the coding radius may be similarly defined.
A coding is called \emph{finitary} if $R$ is almost surely finite. When there exists a finitary coding from $Y$ to $X$, we say that $X$ is a \emph{finitary $\Gamma$-factor} of $Y$. When $X$ is a finitary $\Gamma$-factor of $Y$ for some \iid\ process $Y$, we say that $X$ is \emph{$\Gamma$-\ffiid}. When we simply say that $X$ is \ffiid, we implicitly take $ \Gamma$ to be the entire automorphism group of $G$.

Let us make one last remark concerning the issue of the graph on which a certain model lives. For convenience, we always let the \iid\ process live on the vertices of the graph, even when the model itself does not. Take, for instance, a model on the edges of $\Z^d$, i.e., a random field $X=(X_e)_{e \in E(\Z^d)}$. When we say that $X$ is \ffiid, we mean that there is an \iid\ process on the vertex set of $\Z^d$, say $Y=(Y_v)_{v \in \Z^d}$, and a finitary coding from $Y$ to $X$, which is invariant under the automorphism group of $\Z^d$ (which acts naturally on the edges of $\Z^d$).

\paragraph{Organization.}
The rest of the paper is organized as follows. In \cref{sec:general}, we introduce and prove a general result which will be used to prove that certain gradients are \ffiid. We prove our results for the Potts model in \cref{sec:Ising}, and for the beach model in \cref{sec:beach-model}. In \cref{sec:6v}, we define the six-vertex model, introduce its graphical representation, and establish several properties of it, including its coupling with the spin representation of the six-vertex model. In \cref{sec:h_grad}, we prove our coding results for the six-vertex model. We end with a discussion in \cref{sec:open} on open problems and directions for future research.

\paragraph{Acknowledgements.}
We are grateful to Alexander Glazman and Ron Peled for many useful discussions regarding the six-vertex model and for bringing to our attention the connection between this model and the Ashkin--Teller model.
We would also like to thank Raphael Cerf, Hugo Duminil-Copin and Matan Harel for helpful discussions and also Marcin Lis for insightful discussions about graphical representations of spin models. Finally, we would like to thank the anonymous referee for many helpful comments.

\section{A general result}\label{sec:general}

In this section, we prove a general result (\cref{thm:grad_general} below) about the finitary codability of the gradient of independently colored clusters of an edge percolation process. We will later use this general result for the proofs of the main theorems stated in \cref{sec:intro}.

Let us introduce some notation.
Let $G=(\V,E)$ be a transitive locally-finite connected graph on a countable vertex set $\V$ and let $\Gamma$ be the automorphism group of $G$. Let $\omega \in \{0,1\}^E$ be an edge percolation configuration on $G$. We often identify $\omega$ with the subset $\{ e \in E : \omega_e=1\}$, which may in turn be identified with the subgraph $(\V,\omega)$ of $G$ induced by it.
A cluster of $\omega$ is a connected component in the graph $(\V,\omega)$. We denote by $\cC(\omega)$ the collection of clusters of $\omega$.
For a vertex $u$, we denote by $C_u(\omega) \in \cC(\omega)$ the cluster containing~$u$.
When $\omega$ has a unique infinite cluster (as will always be the case here), we denote it by $C_\infty(\omega)$.

\subsection{A rooted tree of clusters as a finitary factor}

In this section, we prove a general result about the existence of a tree of clusters with certain properties in any percolation process with a unique infinite cluster.

Let us consider the subset $\Omega$ of $\{0,1\}^E$ consisting of all edge percolation configurations having a unique infinite cluster, i.e.,
\begin{equation}\label{eq:Omega-def}
\Omega := \big\{ \omega \in \{0,1\}^E : \omega\text{ has exactly one infinite cluster} \big\} .
\end{equation}
A \emph{cluster-tree} of $\omega \in \Omega$ is a rooted tree on the vertex set $\cC(\omega)$ whose root corresponds to the unique infinite cluster $C_\infty(\omega)$ and is the only node in the tree with an infinite degree.
We note that the automorphism group of~$G$ acts on the space of cluster-trees in a natural way: if $\gamma$ is an automorphism of $G$ and $T$ is a cluster-tree of $\omega \in \Omega$, then $\gamma T$ is a cluster-tree of $\gamma \omega$ (i.e., a tree on vertex set $\cC(\gamma \omega) = \gamma \cC(\omega)$) satisfying that $\{\gamma C,\gamma C'\} \in \gamma T$ if and only if $\{C,C'\} \in T$ for any $C,C' \in \cC(\omega)$.

A \emph{cluster-tree factor map} is a measurable equivariant function which maps every $\omega \in \Omega$ to a cluster-tree on $\omega$ (the space of cluster-trees can be endowed with a natural $\sigma$-algebra).
Intuitively, such a map will be finitary if the rule governing how a finite cluster selects its parent is local in the sense that it can be described via an exploration process which is guaranteed to terminate after finitely many steps. With the goal of defining this finitary property precisely, we now give some definitions.

Given a function $g$ on $\Omega$ and a configuration $\omega \in \Omega$, we say that a set $W \subset \V$ is a \emph{witness} for $g(\omega)$ if $g(\omega)=g(\omega')$ for any $\omega' \in \Omega$ that coincides with $\omega$ on the edges incident to $W$. We note that if $W$ is a witness for $g(\omega)$, then so is any set containing $W$ (and also $\V$ is always a witness). We stress that the set of witnesses for $g(\omega)$ depends on the pair $(g,\omega)$.
We say that $g(\omega)$ can be \emph{found in a finitary manner} if there is a \emph{finite} witness for $g(\omega)$.

Let us give an example. Consider the map $g$ defined by $\omega \mapsto \1_{\{C_v(\omega)\text{ is infinite}\}}$, where $v$ is a fixed vertex.
Then $g(\omega)$ can be found in a finitary manner if and only if the cluster of $v$ in $\omega$ is finite. Indeed, if the cluster of $v$ in $\omega$ is finite, then its vertex set is a witness for $g(\omega)$. On the other hand, if the cluster of $v$ is infinite, then there is no finite witness for $g(\omega)$. Indeed, if $W$ is a finite set, then by closing all edges on the boundary of a sufficiently large ball around $v$, we get a configuration $\omega'$ with $g(\omega')=0$ and at least one infinite cluster, and after closing all but one of these clusters, we get a configuration $\omega'' \in \Omega$ with $g(\omega'')=0$. Since $\omega''$ coincides with $\omega$ near $W$, this shows that $W$ cannot be a witness for $g(\omega)$.

Let us return to our discussion on cluster-tree factor maps.
The above example shows that, given a percolation process, the random field $(\1_{\{C_v\text{ is infinite}\}})_{v \in \V}$ is in general not a finitary factor of the percolation process (in the usual sense). In particular, we also cannot determine in a finitary manner the distance in the cluster-tree from a given finite cluster to the root. Instead, we aim to find the shortest path in the cluster-tree between the clusters of two given vertices, modulo the information of the cluster corresponding to their lowest common ancestor. Since it is possible to determine in a finitary manner whether two given vertices are in the same cluster or not (due to the uniqueness of the infinite cluster), this will allow us to circumvent the aforementioned issue. To describe the function $g$ which encodes the relevant information, we proceed to give the necessary notation.

Let $\cT$ be a cluster-tree factor map and let $\omega \in \Omega$. In the definitions below, we suppress $\omega$ in the notation for clarity (e.g., $\cT=\cT(\omega)$, $\cC=\cC(\omega)$ and so on).
For a finite cluster $C \in \cC$, we denote by $\cP(C)$ the parent of $C$ in $\cT$. When $C$ is at distance at least $k$ from the root in $\cT$, we denote by $\cP_k(C)$ the $k$-th parent of $C$ in $\cT$. In particular, $\cP_0(C)=C$, $\cP_1(C)=\cP(C)$ and $\cP_k(C) = \cP(\cP_{k-1}(C))$.
For two clusters $C,C' \in \cC$, we denote by $\cA(C,C')$ the lowest common ancestor of $C$ and $C'$ in~$\cT$.
We also write $\cA(u,v)$ as shorthand for $\cA(C_u,C_v)$.
Let $N_{u,v}$ denote the distance between $C_u$ and $\cA(u,v)$ in $\cT$.
Note that $N_{u,v}$ is not the same thing as $N_{v,u}$. For instance, if $C_v=\cP(C_u)$, then $N_{u,v}=1$ and $N_{v,u}=0$. More generally, we have that $\cP_{N_{u,v}}(C_u) = \cA(u,v)$.
We also write $\cP_k(u)$ as shorthand for $\cP_k(C_u)$, and set $\cP_{-1}(u)=\{u\}$. Note that $\cP_k(u)$ is finite for all $k \ge -1$ strictly less than the distance between $C_u$ and the root in $\cT$.
We refer the reader to \cref{fig:cluster-tree} for an illustration of some of these notions.

We are now ready to define the function $g$ of interest.
Let $u,v \in \V$ be two vertices, and let $g(\omega)$ be the collection of objects: $N_{u,v}$ and $N_{v,u}$ and the two sequences of finite clusters $(\cP_0(u),\dots,\cP_{N_{u,v}-1}(u))$ and $(\cP_0(v),\dots,\cP_{N_{v,u}-1}(v))$. Note that $g$ depends implicitly on $\cT$. Note also that $\cA(u,v)$ is not included, so that if, say, $\cA(u,v)=C_u$ then $N_{u,v}=0$ and the first sequence is empty.
Furthermore, note that a witness for $g(\omega)$ must be large enough to determine the clusters in $(\cP_0(u),\dots,\cP_{N_{u,v}-1}(u))$ and $(\cP_0(v),\dots,\cP_{N_{v,u}-1}(v))$ (it must contain them). 
Let $R_{u,v}$ be the minimal $r$ such that the union of the two balls of radius $r$ around $u$ and $v$ is a witness for $g(\omega)$.
Thus, $g(\omega)$ can be found in a finitary manner if and only if $R_{u,v}<\infty$.

Let $L_{u,r}$ denote the minimal $\ell \ge r$ such that any two vertices $v,w \in C_\infty$ contained in the ball of radius $r$ around $u$, are connected in $\omega$ within the ball of radius $\ell$ around $u$. Recall the definition of $\Omega$ from~\eqref{eq:Omega-def}.

\begin{figure}
\centering
\includegraphics[scale=0.85,trim={0cm 0cm 0cm 2cm},clip]{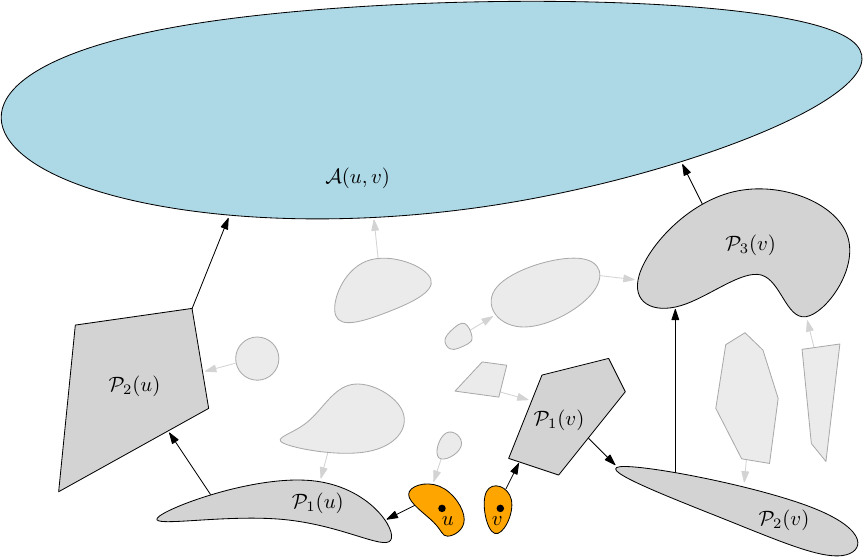}
\caption{An illustration of the cluster-tree construction. Each finite cluster has an arrow pointing to its parent. Given two vertices $u$ and~$v$, one may find in a finitary manner the paths (in dark gray) in the cluster-tree from the clusters of $u$ and $v$ (in orange) to their lowest common ancestor (in blue).}
\label{fig:cluster-tree}
\end{figure}

\begin{prop}\label{prop:tree}
There exists a cluster-tree factor map $\cT$ such that $R_{u,v}$ is finite for every $\omega \in \Omega$ and $u,v \in \V$.
Moreover, if $\omega$ is a random variable that belongs to $\Omega$ almost surely, and there exist constants $c,c',a>0$ such that, for all $u \in \V$ and $r \ge 1$,
\begin{equation}\label{eq:tree-exp-decay-cond}
\P(r \le \diam C_u < \infty) + \P(\dist(C_u,C_\infty) \ge r) + \P(L_{u,r} \ge ar) \le c'e^{-cr} ,
\end{equation}
then $R_{u,v}$ has exponential tails for any $u,v \in \V$.
\end{prop}

We remark that the map $\cT$ is universal in the sense that it does not depend on the law of the random percolation process, but rather it is a single deterministic map which may be applied to any percolation process (even a non-invariant one) having a unique infinite cluster. In fact, as will be clear from the construction, the map is even universal with respect to the underlying graph $G$, in the sense that one does not need to know the structure of the entire graph, only of that part which is revealed during the exploration. However, we do not use this and hence do not make this precise. We also remark that the proposition extends to quasi-transitive graphs.

In the proof below, given two clusters $C$ and $C'$, we write $\dist_\cT(C,C')$ for the distance between $C$ and $C'$ as nodes in the tree $\cT$, and we write $\dist(C,C')$ for the distance between $C$ and $C'$ as subsets in the graph $G$, namely, $\min_{u \in C, v \in C'} \dist(u,v)$.
We denote the ball of radius $r$ around $u$ by $B_r(u)$ and also denote $B_r(U) = \bigcup_{u \in U} B_r(u)$ for a set $U \subset \V$.

\begin{proof}
Let $\omega \in \Omega$.
We begin by defining the cluster-tree $\cT=\cT(\omega)$. Since the root of $\cT$ must be $C_\infty$, we only need to describe how to determine the parents of finite clusters.
Let $C \in \cC$ be a finite cluster.
For $i \ge 1$, let $V_i(C)$ denote the largest-diameter cluster intersecting $B_i(C)$, and set $V_i(C):=\emptyset$ if there is a tie.
The parent of $C$ is defined to be $\cP(C):=V_k(C)$, where $k=k(C)$ is the smallest index such that $\diam V_k(C) \ge k \ge 5\diam C$. Note that such a $k$ necessarily exists since $V_i(C)=C_\infty$ for all $i \ge \dist(C,C_\infty)$ and $\dist(C,C_\infty)<\infty$ as $\omega$ has a unique infinite cluster by assumption. See \cref{fig:cluster-tree2}. Note for later use that
\[ k(C) = \max\{ 5\diam C, \dist(C,\cP(C)) \} .\]

We have thus defined a parent $\cP(C)$ for every finite cluster $C$. Let us now show that $\cT$ is a tree. It is clear that there are no cycles, since the diameter of $\cP(C)$ is strictly larger than that of $C$. Hence, we only need to show that the graph is connected. To this end, we must show that $C_\infty$ is an ancestor of every finite cluster, or equivalently, that $\dist_\cT(C_u,C_\infty)$ is finite for every $u \in \V$.
This follows easily from the claim that
\begin{equation}\label{eq:tree-dist}
\sum_{i=0}^{\dist_\cT(C_u,C_\infty)-2}
\Big[
\diam \cP_i(u) + k(\cP_i(u)) \Big] \le 3 \dist(C_u,C_\infty) .
\end{equation}
We note that~\eqref{eq:tree-dist} is vacuous unless $\dist_\cT(C_u,C_\infty) \ge 2$. We also note that it says nothing about the diameter of the largest finite ancestor cluster of $C_u$ (the one just before the root), nor about its $k$.
Finally, we note that~\eqref{eq:tree-dist} implies that all the finite ancestors of $u$, namely, $\cP_1(u),\dots,\cP_{\dist_\cT(C_u,C_\infty)-1}(u)$, are at distance at most $3\dist(C_u,C_\infty)$ from $u$, since $\dist(C,\cP(C)) \le k(C)$ for any finite cluster $C$, and for any $1 \le i \le \dist_\cT(C_u,C_\infty)-1$,
\begin{equation}\label{eq:triangle}
\dist(u,\cP_i(u)) \le \diam C_u + \dist(C_u,\cP_i(u)) \le \sum_{j=0}^{i-1} \Big[
\diam \cP_j(u) + \dist(\cP_j(u),\cP_{j+1}(u)) \Big] .
\end{equation}

Towards proving~\eqref{eq:tree-dist}, suppose that $n := \dist_\cT(C_u,C_\infty) - 2$ is non-negative and denote $P_i := \cP_i(u)$ and $d_i := \diam P_i$ for $0 \le i \le n+2$. Also denote $k_i := k(P_i)$ for $0 \le i \le n+1$. Note that $P_{n+1}$ is finite and $P_{n+2}=C_\infty$. Observe that, by construction, for every $0 \le i \le n+1$, we have
\[ d_{i+1} \ge k_i = \max\{5d_i, \dist(P_i,P_{i+1})\} \qquad\text{and}\qquad P_{i+1} = V_{k_i}(P_i) .\]
Then
\begin{align*}
  d_0+k_0+\dots+d_n+k_n
  &\le \tfrac 65 (k_0+\dots+k_n) \\ 
  &\le \tfrac 65 (\tfrac 1{5^n} + \cdots + \tfrac 1{5^3} + \tfrac 1{5^2} + \tfrac 15 + 1)k_n
  \le \tfrac32 k_n.
\end{align*}
Moreover, since $P_{n+1}$ has the largest diameter among all clusters intersecting $B_{k_n}(P_n)$ and since $P_{n+1} \neq C_\infty$, we have that $\dist(P_n,C_\infty) > k_n$.
Hence,
\begin{align*}
 k_n
  &< \dist(C_\infty,P_n) \\
  &\le \dist(C_\infty,P_0) + \diam P_0 + \dist(P_0,P_1) + \cdots + \diam P_{n-1}+ \dist(P_{n-1},P_n) \\
  &\le \dist(C_\infty,C_u) + d_0 + k_0 + \cdots + d_{n-1} + k_{n-1} \\
  &\le \dist(C_\infty,C_u) + \tfrac12 k_n ,
\end{align*}
so that $k_n \le 2\dist(C_u,C_\infty)$.
Therefore, $d_0+k_0+\dots+d_n+k_n \le \frac32 k_n \le 3\dist(C_u,C_\infty)$, thus proving~\eqref{eq:tree-dist}.

\begin{figure}[t!]
\centering
\captionsetup{width=0.96\textwidth}
\includegraphics[scale=0.8,trim={1cm 0cm 0cm 0.5cm},clip]{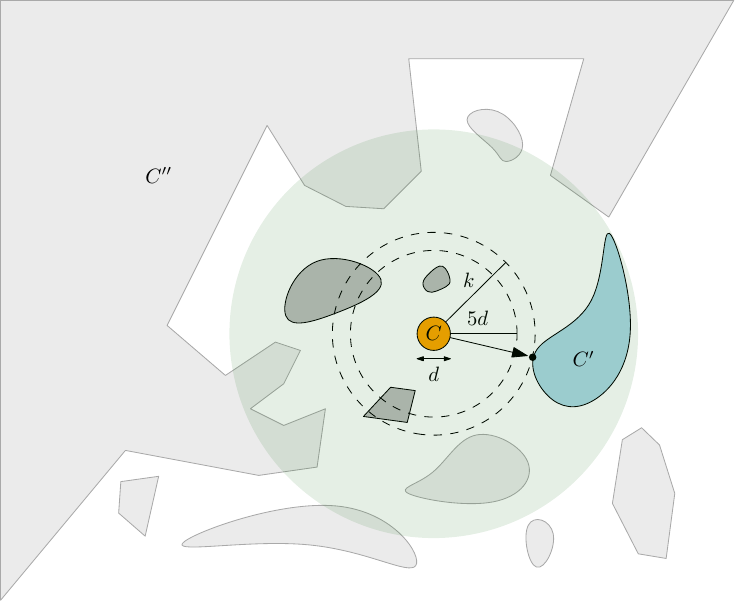}~~~%
\includegraphics[scale=0.8,trim={1cm 0cm 0cm 0.5cm},clip]{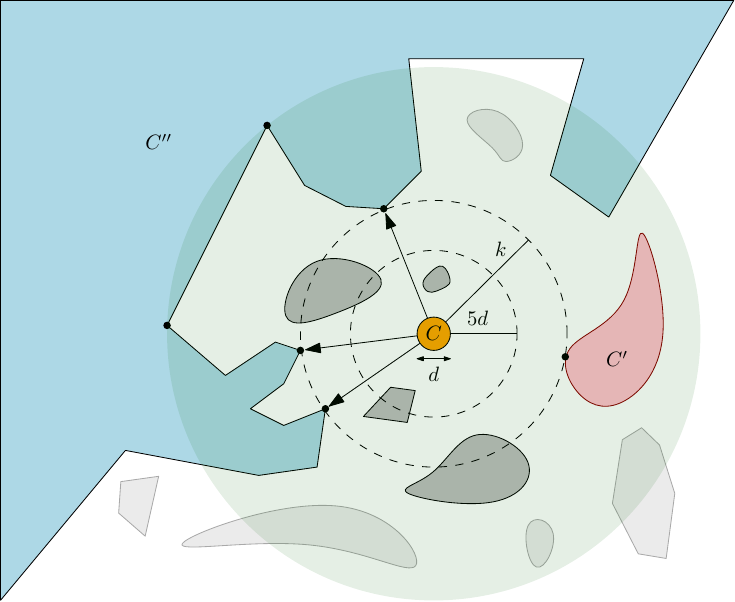}
\caption{An illustration of how the parent of a finite cluster $C$ is found in a finitary manner. In both situations depicted, $d=\diam C$, $k=k(C)$, the parent of $C$ is shown in blue and a ball around $C$ witnessing its parent is shown in light green. The dark gray clusters were tested as potential parents of $C$ at a prior stage (before reaching distance $k$), but did not satisfy the required criteria. The light gray clusters do not need to be tested as the parent of $C$ was already found before they were reached. The red cluster (appearing only in the situation on the right) was tested at the same stage $k$ as when the parent of $C$ was found, but was not the largest-diameter cluster at that stage and so was not chosen as the parent.
\textit{Left:} $C'$ is the parent of $C$.
\textit{Right:} The same cluster configuration as on the left, except that $C'$ has moved slightly to the right. Here $C''$ is the parent of $C$.}
\label{fig:cluster-tree2}
\end{figure}

It remains to establish the desired finitaryness property and the moreover part. Fix two distinct vertices $u,v \in \V$. Before giving the details, let us explain heuristically what is happening. See also \cref{fig:cluster-tree2}. We start with $u$. We begin exploring the cluster of $u$ on larger and larger balls around $u$. If at some point we see that its cluster is finite, then we proceed to find its parent cluster. To find its parent, starting with $i=5\diam C_u$, we partially explore the clusters intersecting $B_i(C_u)$ as follows. Simultaneously for each $w$ in this set, we explore the cluster of $w$ in increasing balls around $C_u$. Eventually we will discover which $w$ belong to distinct clusters (i.e., the connectivities between different $w$), and we can continue exploring until we have seen all clusters in their entirety except perhaps one (which may or may not be the infinite cluster). At this point, in order to determine the largest-diameter cluster, we need only determine whether the unknown cluster has a larger diameter than the others. By increasing the radius of exploration if necessary, this may be determined. By further increasing the radius if necessary, we can also determine whether this cluster has diameter at least $i$ (which was perhaps already known). Finally, we now know whether there is a largest-diameter cluster (or rather a tie), and if so, whether it has diameter at least $i$. We thus know whether $k(C_u)=i$ or not, i.e., whether to increase $i$ by one and repeat the above, or whether to stop. In the latter case, if the parent is a finite cluster, we may have already discovered it completely, but if it is the infinite cluster, we surely have not. Either way, we have found a vertex $w$ for which we know that the parent of $C_u$ is $C_w$ (even if we may not know the shape of this latter cluster). We then continue in the same manner, namely, we begin exploring $C_w$ from the vertex $w$ which we have already found, and if at some point we see that $C_w$ is finite, we proceed to find its parent, and repeat. In parallel, we do the same for $v$. At some point both $u$ and $v$ will have discovered pieces of $\cA(u,v)$ and each will know that the piece it has discovered is part of an ancestor cluster. If $\cA(u,v)$ is finite, then at some point both will discover the entire cluster. Otherwise, at some point we will see that the two pieces are connected to each other (since the infinite cluster is unique). Either way, we will know that we have reached the common ancestor, so that we may stop exploring. This shows that $R_{u,v}$ is finite. Of course, in order to deduce the moreover part, we must have sufficiently good control on $R_{u,v}$. The three elements for this are: control on the size of finite clusters (as we need to explore enough in order to be sure that a certain cluster is the parent of another), control on the distance to the infinite cluster (as this is what is ensuring that the ancestors of $u$ and $v$ do not drift to far away from each other), and control on the connectivity of the infinite cluster (as we need to explore enough in order to see that the two pieces discovered by $u$ and $v$ are indeed connected to each other in the case that $C_\infty$ is their common ancestor).
We now proceed to give the details.

Denote $R_u := \dist(C_u,C_\infty)$ and
\[ M_u := \max_{w \in B_{3R_u}(u)} (\diam C_w) \1_{\{C_w \neq C_\infty\}} \quad\text{and}\quad M'_u := \max_{w \in B_{7R_u+6M_u}(u)} (\diam C_w) \1_{\{C_w \neq C_\infty\}} .\]
We explain the need for the above ``two-step iteration'' (note that $M'_u$ uses $M_u$ in its definition): while $M_u$ controls the sizes of the ancestor clusters of $u$ and other clusters nearby, it does not control the sizes of clusters nearby the last finite ancestor $\cP_{\dist_\cT(C_u,C_\infty)-1}(u)$, which will be needed in order to witness the fact that its parent is $C_\infty$ and not some other large nearby cluster. For this reason, we also require $M'_u$, which provides this control.
Similarly define $R_v$, $M_v$ and $M'_v$.
Finally, define
\[ L_* := L_{u,\dist(u,v) + 7(R_u+R_v)+6(M_u+M_v)} .\]
We note that there is a slight asymmetry between $u$ and $v$ in the definition of $L_*$, but this is not important (other definitions would also work).
Since all these variables are finite for every $\omega \in \Omega$, the first part of the proposition will follow once we show that
\begin{equation}\label{eq:tree-coding-radius}
R_{u,v} \le 7(R_u+R_v)+6(M_u+M_v) + M'_u+M'_v +L_* .
\end{equation}
Before establishing this, let us show how it yields the moreover part of the proposition.
To this end, suppose that $\omega \in \Omega$ is random and that~\eqref{eq:tree-exp-decay-cond} holds.
Since $R_u$ and $(\diam C_u)\1_{\{C_u \neq C_\infty\}}$ have exponential tails, so does $M_u$. More precisely, letting $b>0$ be such that $|B_{3br}(u)| e^{-cr} \le e^{-cbr}$ (which exists since $G$ has bounded degree), we have that
\begin{align*}
\P(M_u \ge r) &\le \P(R_u \ge br) + |B_{3br}(u)| \cdot \max_{w \in B_{3br}(u)} \P(r \le \diam C_w < \infty) \le 2c'e^{-cbr} .
\end{align*}
By a similar argument, we get that $M'_u$ has exponential tails. To see that $L_*$ has exponential tails, note that
\[ \P(L_* \ge ar) \le \P(\dist(u,v)+7(R_u+R_v)+6(M_u+M_v) \ge r) + \P(L_{u,r} \ge ar) \]
and that both terms decay exponentially in $r$. Thus, the right-hand side of~\eqref{eq:tree-coding-radius} has exponential tails, showing that $R_{u,v}$ does as well.

It remains to prove~\eqref{eq:tree-coding-radius}. By definition of $R_{u,v}$, this means we need to show that $N_{u,v}$, $N_{v,u}$, $(\cP_i(u))_{i<N_{u,v}}$ and $(\cP_i(v))_{i<N_{v,u}}$ are witnessed by
\[ W := B_{7(R_u+R_v)+6(M_u+M_v) + M'_u+M'_v +L_*}(\{u,v\}) .\]
Set $n := \dist_\cT(C_u,C_\infty)$ and $n' := \dist_\cT(C_v,C_\infty)$.

Let us first show that for any $0 \le i < n$, $W$ witnesses the event $\{\dist_\cT(C_u,C_\infty) \ge i\}$ and the entire cluster $\cP_i(u)$. To emphasize, the latter means that, for any $\omega' \in \Omega$ which agrees with $\omega$ on the edges incident to $W$, it holds that $\dist_{\cT(\omega')}(C_u(\omega'),C_\infty(\omega')) \ge i$ and $\cP_i(u;\omega') = \cP_i(u;\omega)$.
Suppose that we have already shown that $W$ witnesses $\cP_{i-1}(u)$, and let us show that it also witnesses $\cP_i(u)$.
Note that $C_u$ and $\cP_i(u)$ are finite since $N_{u,v}>i$.
Recall that, by definition, $\cP_i(u) = V_k(\cP_{i-1}(u))$, where $k := k(\cP_{i-1}(u))$ is the minimal number such that $\diam V_k(\cP_{i-1}(u)) \ge k \ge 5\diam \cP_{i-1}(u)$.
By definition, every cluster intersecting $B_k(\cP_{i-1}(u))$, other than $\cP_i(u)$, has diameter strictly less than $d := \diam \cP_i(u)$.
It is straightforward to verify that $B_{k+d}(\cP_{i-1}(u))$ witnesses $k$ and $\cP_i(u)$. Thus, to deduce that $W$ witnesses $\cP_i(u)$, it remains only to show that $B_{k+d}(\cP_{i-1}(u)) \subset W$, for which it suffices to show that $\dist(u,\cP_{i-1}(u)) + \diam \cP_{i-1}(u) + k+d \le 3R_u+M_u$. Indeed, $\dist(u,\cP_{i-1}(u)) + \diam \cP_{i-1}(u) + k \le 3R_u$ by~\eqref{eq:tree-dist} and~\eqref{eq:triangle}, and $d \le M_u$ follows from the definition of $M_u$ and since $\dist (u,\cP_i(u)) \le 3R_u$ by~\eqref{eq:tree-dist} and~\eqref{eq:triangle}.

We similarly have that for any $0 \le j < n'$, $W$ witnesses the event $\{\dist_\cT(C_v,C_\infty) \ge j\}$ and $\cP_j(v)$. 

It remains to show that $W$ witnesses $N_{u,v}$ and $N_{v,u}$.
Observe that this already follows from the above in the case when $\cA(u,v) \neq C_\infty$. Indeed, in this case, $N_{u,v}$ is the smallest $0 \le i < n$ such that $\cP_i(u) = \cP_j(v)$ for some $0 \le j < n'$, and similarly for $N_{v,u}$. When $\cA(u,v) = C_\infty$, we cannot expect to actually find these sets in a finitary manner.
Note however that we do not actually need to know the sets $\cP_i(u)$ and $\cP_j(v)$ themselves, but rather only whether they are equal or not. Instead, we show that $W$ witnesses the existence of two numbers $i$ and $j$ such that $\cP_i(u)=\cP_j(v)$ (though it does not witness what this common set is). From this it is then clear that $W$ witnesses $N_{u,v}$ and $N_{v,u}$, thereby completing the proof of~\eqref{eq:tree-coding-radius}.
To do this, we shall show that there exist two subsets $A_u,A_v \subset \V$ such that $W$ witnesses the event $\{ A_u$ is contained in an ancestor cluster of $u$, $A_v$ is contained in an ancestor cluster of $v$, and $A_u$ and $A_v$ are connected$\}$.

Let us now try to repeat the above argument in the case when $i=n$. Since we cannot find $\cP_n(u)=C_\infty$ in a finitary manner, we aim to find a set $A_u$ as above, that is, a set which is guaranteed to belong to $\cP_n(u;\omega')$ for any $\omega'$ which agrees with $\omega$ on the edges incident to $W$ (though there is no guarantee that $\cP_n(u;\omega')=C_\infty(\omega')$; indeed, it is not possible to guarantee this).
Similarly to before, $C_\infty = \cP_n(u) = V_k(\cP_{n-1}(u))$, where $k:=k(\cP_{n-1}(u))$ is the minimal number such that $\diam V_k(\cP_{n-1}(u)) \ge k \ge 5\diam \cP_{n-1}(u)$.
Let $d$ be the largest diameter of a finite cluster intersecting $B_k(\cP_{n-1}(u))$.
Let $\ell := L_{w,k+\diam \cP_{n-1}(u)}$, where $w$ is a vertex of $\cP_{n-1}(u)$ closest to $u$. Note that this definition ensures that $B_k(\cP_{n-1}(u)) \cap C_\infty$ (which is necessarily non-empty, but may contain more than one vertex) is contained in a single connected component of $B_{d+\ell}(\cP_{n-1}(u)) \cap C_\infty$, and that this component has diameter strictly larger than $d$. Let $A_u$ denote this component.
It is straightforward to verify that $B_{k+d+\ell}(\cP_{n-1}(u))$ witnesses $k$ and $A_u$. Similarly to before, to deduce that $W$ witnesses $A_u$, we need only show that
\[ \dist(u,\cP_{n-1}(u)) + \diam \cP_{n-1}(u) + k+d+\ell \le 7R_u+6M_u+M'_u+L_* .\]
Let us give several inequalities which easily imply this. First, $\dist(u,\cP_{n-1}(u)) \le 3R_u$ by~\eqref{eq:tree-dist} and~\eqref{eq:triangle}. Second, $\diam \cP_{n-1}(u) \le M_u$ by definition of $M_u$. Third, by~\eqref{eq:tree-dist} and~\eqref{eq:triangle},
\begin{align*}
k &= \max\{5\diam \cP_{n-1}(u), \dist(\cP_{n-1}(u),C_\infty)\} \\
  &\le \max\{5M_u,\dist(C_u,C_\infty)+\diam C_u + \dist(C_u,\cP_{n-1}(u))\} \le \max\{5M_u,4R_u\} .
\end{align*}
Fourth, since $\dist(w,u) = \dist(u,\cP_{n-1}(u)) \le 3R_u$, we have $\ell \le L_{u,7R_u+6M_u} \le L_*$.
Fifth, since $\dist(u,\cP_{n-1}(u))+\diam \cP_{n-1}(u)+k \le 7R_u+6M_u$, we have $d \le M'_u$ by definition of $M'_u$.
We also note that $\dist(u,A_u) \le \dist(u,\cP_{n-1}(u)) + \diam \cP_{n-1}(u) + k \le 7R_u + 6M_u$.

The argument for finding $A_v$ is analogous ($L_*$ is the only non-symmetric term, and so we only note that $L_{v,7R_v+6M_v} \le L_*$ holds). It remains to show that $W$ witnesses that $A_u$ and $A_v$ belong to the same cluster. To this end, it suffices to show that $A_u$ and $A_v$ are connected inside $W$.
This will follow from the definition of $L_*$ once we show that both $A_u$ and $A_v$ are at distance at most $\dist(u,v)+7R_u+7R_v+6M_u+6M_v$ from~$u$.
Indeed, this follows from $\dist(u,A_u) \le 7R_u+6M_u$ and $\dist(v,A_v) \le 7R_v+6M_v$, which we have just shown. This completes the proof of~\eqref{eq:tree-coding-radius} and hence also of the proposition.
\end{proof}

\subsection{Gradient of spins as a finitary factor}
\label{sec:general-gradient}

Fix an integer $q \ge 2$ and let $\omega$ be a random percolation configuration in $\{0,1\}^E$. Construct a random spin configuration $\sigma \in \{0,\dots,q-1\}^\V$ by assigning a spin to each vertex so that, conditionally on $\omega$,
\begin{itemize}
 \itemsep 0pt
 \item spins belonging to the same cluster are equal,
 \item spins belonging to different clusters are independent,
 \item spins belonging to finite clusters are distributed uniformly in $\{0,\dots,q-1\}$,
 \item spins belonging to an infinite cluster are 0 (or any other fixed value).
\end{itemize} 
For an oriented edge $e = (u,v)$, define $(\nabla \sigma)_e = \sigma_v - \sigma_u$ mod $q$.

\begin{thm}\label{thm:grad_general}
Suppose that $\omega$ is a random percolation process on $G$ which
almost surely has a unique infinite cluster. Let $q \ge 2$ and define a spin configuration $\sigma$ as above.
Then $\nabla \sigma$ is a finitary factor of $(\omega,\xi)$, where $\xi$ is an \iid\ process independent of $\omega$. In particular, if $\omega$ is \ffiid, then so is $\nabla \sigma$. Moreover, if $\omega$ satisfies~\eqref{eq:tree-exp-decay-cond} and is \ffiid\ with a coding radius having exponential tails, then so is $\nabla \sigma$.
\end{thm}

\begin{proof}
Let $\xi=(U_v,Y_v)_{v \in \V}$ be an \iid\ process, independent of $\omega$, where $U_v \sim \text{Unif}[0,1]$ and $Y_v \sim \text{Unif}\{0,1,\ldots,q-1\}$ are independent. Let $\cT$ be the cluster-tree factor map from \cref{prop:tree}. When $C$ is finite cluster, we define $Y_C$ to be the variable $Y_v$ where $v$ is the vertex in $C$ with minimal $U_v$. 
We stress that the spin $Y_C$ will not correspond to the spin of $C$ in $\sigma$, but rather indicates the spin relative to its parent cluster.
To define this precisely, we associate a spin to each edge of the tree $\cT(\omega)$ by setting $Y_{C,\cP(C)} = Y_C$ if $C$ is a finite cluster. When $C'$ is an ancestor of $C$, we define $Y_{C,C'}$ to be the sum of spins along the edges from $C$ to $C'$. In particular, $Y_{C,C}=0$ for any cluster $C$, including the infinite cluster $C_\infty$.

For every vertex $v \in \V$, we define 
\begin{equation*}
\sigma_v  = Y_{C_v,C_\infty}  \mod q.		
\end{equation*}
We claim that $\sigma$ has the desired distribution. To see this, note that $\sigma_u=\sigma_v$ whenever $u$ and $v$ are in the same cluster, that $\sigma_v=0$ whenever $v \in C_\infty$, and that for any finite collection of finite clusters $C_1,\dots,C_n$, the variables $\{Y_{C_i,C_\infty}\}_{1 \le i \le n}$ are independent and uniformly distributed (mod $q$). Indeed, note that some cluster, say $C_1$, will have no descendants in $C_1,\dots,C_n$, and it is clear that in this case, $Y_{C_1,C_\infty}$ is uniform conditioned on $(Y_C)_{C \neq C_1}$.

Note that this already shows that $\sigma$ is a (non-finitary) factor of $(\omega,\xi)$. This representation allows for a natural way to interpret $\nabla \sigma$, namely, for every oriented edge $e = (u,v)$, we have
\begin{align*}
\nabla \sigma_e
 &= Y_{C_v,C_\infty} - Y_{C_u,C_\infty} \mod q \\
 &= Y_{C_v,\cA(u,v)} - Y_{C_u,\cA(u,v)} \mod q .
\end{align*}
Indeed, this is straightforward from the definitions.

It remains to show that $\nabla \sigma$ is a finitary factor of $(\omega,\xi)$. As we have mentioned, $\sigma$ a factor of $(\omega,\xi)$, and hence, $\nabla\sigma$ is also a factor of $(\omega,\xi)$. Thus, we need only show that the latter factor is finitary, i.e., that $\nabla\sigma_e$ can be determined in a finitary manner for any oriented edge $e = (u,v)$. By the formula above, in order to determine $\nabla \sigma_e$, it suffices to determine $Y_{C_u,\cA(u,v)}$ and $Y_{C_v,\cA(u,v)}$.
Let us explain how $Y_{C_u,\cA(u,v)}$ can be determined in a finitary manner (the argument for $Y_{C_v,\cA(u,v)}$ being the same).
By definition,
\[ Y_{C_u,\cA(u,v)} = Y_{\cP_0(u),\cP_{N_{u,v}}(u)} = \sum_{i=0}^{N_{u,v}-1}Y_{\cP_i(u), \cP_{i+1}(u)}  =  \sum_{i=0}^{N_{u,v}-1} Y_{\cP_i(u)} .\]
By \cref{prop:tree}, the map $\cT$ has the property that $R_{u,v}$ is almost surely finite. In particular, we can almost surely find $N_{u,v}$ and $(\cP_i(u))_{0 \le i < N_{u,v}}$ in a finitary manner. Hence, it suffices to show that we can determine $Y_C$ in a finitary manner for every $C \in \{\cP_i(u)\}_{0 \le i < N_{u,v}}$. This is clear from the definition of $Y_C$ and the fact that these clusters are almost surely finite.

Finally, towards showing the moreover part, suppose that $\omega$ satisfies~\eqref{eq:tree-exp-decay-cond} and is \ffiid\ with a coding radius having exponential tails. Then, by \cref{prop:tree}, the coding radius needed to determine $N_{u,v}$ and $(\cP_i(u))_{0 \le i < N_{u,v}}$ from $\omega$ has exponential tails (note that $\dist(u,v)=1$ here, so that the distinction between balls centered around $u$ or $v$ is not important). Since this coding radius is always large enough so that the ball of this radius around $u$ completely contains the clusters $\{\cP_i(u)\}_{0 \le i < N_{u,v}}$, it is easy to see that it also allows to determine $Y_{\cP_i(u)}$ from $(\omega,\xi)$ for every $0 \le i < N_{u,v}$. Therefore, the coding radius for determining $\nabla\sigma_e$ from $(\omega,\xi)$ has exponential tails. Since $\omega$ is \ffiid\ with exponential tails, and since the composition of finitary factors with exponential tails is also such (see~\cite[Lemma~3.3]{harel2018finitary}), we conclude that $\nabla\sigma$ is \ffiid\ with exponential tails.
\end{proof}

\begin{remark}
The proof of \cref{thm:grad_general} easily extends to the situation in which the spin space $\{0,\dots,q-1\}$ is replaced with any finite group.
\end{remark}

\section{The Ising and Potts models}\label{sec:Ising}

The (ferromagnetic) Potts model on $\Z^d$ with $q \in \{2,3,\ldots\}$ states and inverse temperature $\beta \ge 0$ is defined as follows. Given a finite set $V \subset \Z^d$ and a configuration $\tau \in \{0,\dots,q-1\}^{\Z^d}$, the finite-volume Gibbs measure in $V$ with boundary condition $\tau$ is the probability measure $P^\tau_V$ on $\{0,\ldots,q-1\}^{\Z^d}$ defined by
\begin{equation}\label{eq:Potts-spec}
P^\tau_V(\sigma) = \frac{1}{Z^\tau_V} \cdot e^{\beta H_V(\sigma)} \cdot \1_{\{\sigma=\tau\text{ outside }V\}} , \qquad\text{where }H_V(\sigma) = \sum_{\substack{\{u,v\} \in E(\Z^d)\\\{u,v\} \cap V \neq \emptyset}} \1_{\{\sigma_u=\sigma_v\}} .
\end{equation}
Here $Z^\tau_V$ (which also depends on $q$ and $\beta$) is a normalization constant.
A \emph{Gibbs measure} for the Potts model is a probability measure $\mu$ on $\{0,\dots,q-1\}^{\Z^d}$ such that a random configuration $\sigma$ distributed according to $\mu$ has the property that, for any finite $V \subset \Z^d$, conditioned on the restriction $\sigma|_{V^c}$, $\sigma$ is almost surely distributed according to $P^\sigma_V$.

Consider the Ising model on $\Z^d$ with $d \ge 2$ -- this is the special case of the Potts model in which $q = 2$. In this case, it is common to let the spin values be $\{-1,1\}$, rather than $\{0,1\}$.
It is well known~(see, e.g., \cite[Theorem~3.1]{georgii2001random} or \cite[pages 189-190 and 204]{liggett2012interacting}) that there exists a critical value $\beta_c(d) \in (0,\infty)$ such that there is a unique Gibbs measure for the Ising model on $\Z^d$  at inverse temperature $\beta<\beta_c(d)$ and multiple such Gibbs measures at inverse temperature $\beta>\beta_c(d)$. It has also been established that there is a unique Gibbs measure at the critical point $\beta=\beta_c(d)$ (see \cite{yang1952spontaneous,aizenman2015random,aizenman1986critical}).
It is also well known that, when $\beta>\beta_c(d)$, there exist two distinct extremal Gibbs states, called the plus state and the minus state, obtained as limits of $P^\tau_V$ as $V$ increases to $\Z^d$ with the boundary condition $\tau$ being the all plus or all minus configuration.

It was shown by van den Berg and Steif~\cite{van1999existence} that this model (more precisely, the plus or minus state) is ffiid if and only if there is a unique Gibbs measure, i.e., if and only if $\beta \le \beta_c(d)$. We show in this paper that a slight dilution of information makes this model ffiid even when $\beta > \beta_c(d)$. Specifically, we consider here the gradient of the Ising model -- the percolation configuration consisting of all edges whose endpoints have different spins. More precisely, given an Ising spin configuration $\sigma \in \{-1,1\}^{\Z^d}$, we consider the percolation configuration $\nabla \sigma \in  \{0,1\}^{E(\Z^d)}$ defined by
\[ (\nabla \sigma)_{\{u,v\}} = \1_{\{\sigma_u \neq \sigma_v\}} .\]
Thus, a Gibbs measure for the Ising model induces a probability measure on $\{0,1\}^{E(\Z^d)}$ via the map $\sigma \mapsto \nabla \sigma$. We call the percolation measure induced by the plus state of the Ising model the \emph{gradient of Ising}. Note that since the minus state is obtained from the plus state by flipping all the spins, the minus state induces the same percolation measure.

\begin{thm}\label{thm:main_Ising}
Let $d \ge 2$ and $\beta > \beta_c(d)$.
The gradient of Ising on $\Z^d$ at inverse temperature~$\beta$ is \ffiid\ with a coding radius having exponential tails.
\end{thm}

Let us mention a simple consequence of \cref{thm:main_Ising} and the previously stated fact that the Ising model is \ffiid\ for all $\beta \le \beta_c(d)$. It is a general fact that any \ffiid\ random field satisfies the ergodic theorem with an exponential rate of convergence~\cite{bosco2010exponential}. Applied to the gradient of the Ising, this yields a volume-order large deviation estimate for the energy $H_V(\sigma)$ defined in~\eqref{eq:Potts-spec}. This does not require any quantitative information on the coding radius, and hence applies also at criticality. This is the content of the following corollary.

\begin{corollary}\label{cor:energy}
Let $d \ge 2$ and $\beta>0$, and consider the plus state $\mu$ for the Ising model on $\Z^d$ at inverse temperature $\beta$. Denote $m=\mu(\sigma_u=\sigma_v)$, where $u$ and $v$ are adjacent vertices. Let $\Lambda_n$ denote the box $[-n,n]^d$ and let $E_n$ denote the set of edges of $\Z^d$ intersecting $\Lambda_n$. Then, for any $\epsilon>0$ there exists $c>0$ such that
 \[ \mu\left( \left| \tfrac{H_{\Lambda_n}(\sigma)}{|E_n|} - m \right| > \epsilon \right) \le e^{-cn^d} \qquad\text{for all }n \ge 1 .\]
\end{corollary}

Large deviation principles have been shown to hold for various random fields~\cite{follmer1988large,georgii1993large}, including the Ising model, but to the best of our knowledge, the fact that the rate is positive for the energy functional is new.
We remark that a similar result holds for the other models for which we prove a finitary coding result, but we will not state these explicitly for those models. 

\smallskip

Consider now the $q$-state Potts model on $\Z^d$ with $d \ge 2$ and $q \ge 3$.
It is well known~(see, e.g., \cite[Theorem~3.2]{georgii2001random}) that there exists a critical value $\beta_c=\beta_c(q,d) \in (0,\infty)$ such that there is a unique Gibbs measure at inverse temperature $\beta<\beta_c$ and multiple Gibbs measures at inverse temperature $\beta>\beta_c$.
In two dimensions, it is also known that there is a unique Gibbs measure at the critical $\beta=\beta_c$ if and only if $q \le 4$~\cite{duminil2016discontinuity,duminil2017continuity,ray2019short}.
It is also well known that, in any dimension, when multiple Gibbs measures exist, there are at least $q$ such measures (one for each spin value), obtained as limits of $P^\tau_V$ with constant boundary conditions. These measures may be obtained from one another by applying a permutation to the spin values.

It follows from results of Harel and the second author~\cite{spinka2018finitarymrf,harel2018finitary} that this model (more precisely, any constant boundary condition Gibbs measure) is \ffiid\ if and only if there is a unique Gibbs measure.
Thus, as for the Ising model, the Potts model is not \ffiid\ at low temperature $\beta> \beta_c(d,q)$, and the reason for this is similar to the one in the Ising case. We are therefore led to consider a gradient of the Potts model.
One possibility, which is a natural generalization of the gradient in the Ising case, is to consider the percolation configuration consisting of edges whose endpoints have different spins. We instead choose a different (also natural) extension of the definition -- one which preserves more information on the relative spin values at the endpoints of an edge. Specifically, given a Potts configuration $\sigma \in \{0,\dots,q-1\}^{Z^d}$, the gradient of $\sigma$ is a configuration living on $\vec E(\Z^d)$, the \emph{oriented} edges of $\Z^d$, and is defined by
\[ (\nabla\sigma)_{(u,v)} = \sigma_v - \sigma_u \mod q \]
for an oriented edge $(u,v) \in \vec E(\Z^d)$.
A Gibbs measure for the Potts model induces a probability measure on $\{0,\dots,q-1\}^{\vec E(\Z^d)}$ via the map $\sigma \mapsto \nabla \sigma$. We call the measure induced by a constant boundary condition Gibbs state (any constant boundary condition induces the same measure) the \emph{gradient of Potts}. Note that when $q=2$, this is essentially the same as the gradient of Ising, except that the latter has unoriented edges.

The Potts model is closely related to the random-cluster model. We briefly recall this here and refer to the book of Grimmett \cite{grimmett2006random} for a more comprehensive treatment of this model. The random-cluster measure with boundary condition $\tau \in \{0,1\}^{E(\Z^d)}$ and parameters $q>0$ and $p \in [0,1]$ in a finite subset $\Lambda \subset E(\Z^d)$ is given by
\begin{equation}
\P^{\text{FK},\tau}_{\Lambda,p,q} (\omega) = \frac1{Z^{\text{FK},\tau}_{\Lambda,p,q}} \cdot p^{o_\Lambda(\omega)}(1-p)^{c_\Lambda(\omega)} q^{k_\Lambda(\omega)} \cdot \1_{\{\omega=\tau\text{ outside }\Lambda\}}, \label{eq:FK}
\end{equation}
where $o_\Lambda(\omega)$ and $c_\Lambda(\omega)$ are the number of open and closed edges, respectively, of $\omega$ in $\Lambda$, $k_{\Lambda}(\omega)$ is the number of vertex-clusters of $\omega$ intersecting $\Lambda$, and $Z^{\text{FK},\tau}_{\Lambda,p,q}$ is the appropriate partition function. If $\tau$ is specified to be all edges open (resp.\ closed), then the 
resulting measure is called the \emph{wired} (resp.\ \emph{free}) measure.  We denote the wired and the free measure by $\P^{\text{FK},\text{w}}_{\Lambda, p,q}$ and $\P^{\text{FK},\text{f}}_{\Lambda, p,q}$ respectively. The random-cluster measures satisfy several monotonicity properties; of relevance here is the monotonicity in boundary conditions (FKG), namely, opening more edges in $\tau$ stochastically increases $\P^{\text{FK},\tau}_{\Lambda,p,q}$. This implies in particular that the weak limits of $\P^{\text{FK},\text{w}}_{\Lambda, p,q}$ and $\P^{\text{FK},\text{f}}_{\Lambda, p,q}$ as $\Lambda \uparrow \Z^d$ exist (see~\cite[Theorem~4.19]{grimmett2006random}). The two limiting measures, called the wired and free random-cluster measures, are denoted by $\P^{\text{FK},\text{w}}_{p,q}$ and $\P^{\text{FK},\text{f}}_{p,q}$, respectively.

We prove that the gradient of the Potts model is \ffiid\ when the corresponding random-cluster model has a unique Gibbs state (i.e., the free and wired measures coincide), and that this condition is also necessary.

\begin{thm}\label{thm:main_potts}
Let $d \ge 2$ and $q \ge 2$ be integers, let $\beta \ge 0$ and set $p := 1-e^{-\beta}$. The gradient of the $q$-state Potts model on $\Z^d$ at inverse temperature $\beta$ is \ffiid\ if and only if the free and wired random-cluster measures with parameters $q$ and $p$ coincide.
\end{thm}

Let us briefly discuss the condition in the theorem, namely, uniqueness of the Gibbs state for the random-cluster model. It is well known that for any $d \ge 2$ and $q \ge 1$, there is a critical parameter $p_c(q,d) \in (0,\infty)$ for the existence of infinite clusters in this model (see~\cite[Section~5.1]{grimmett2006random}). Actually, for integer $q \ge 2$, we have the relation $p_c(q,d) = 1-e^{-\beta_c(q,d)}$, where $\beta_c(q,d)$ is the critical inverse temperature for the Potts model. It is known that the random-cluster model admits a unique Gibbs state for all $d \ge 2$, $q>1$ and $p<p_c(q,d)$ (see~\cite[Theorem~5.16]{grimmett2006random}). It is believed that it also has a unique Gibbs state for all $d \ge 2$, $q>1$ and $p>p_c(q,d)$ (see~\cite[Conjecture~5.34]{grimmett2006random}).
This is known to be true in two dimensions (see~\cite[Theorem~6.17]{grimmett2006random}), as well as for the FK-Ising model ($q=2$) in all dimensions~\cite{Bod06}.
In general, for any given $d \ge 3$ and $q>1$, this is known to be the case for all high values of $p$, and for all but at most countably many values of $p>p_c(q,d)$ (see~\cite[Theorem~5.33]{grimmett2006random}). 

Let us also mention that the Potts model has a unique Gibbs state (and is thus \ffiid\ as mentioned above) if and only if the free and wired random-cluster measures coincide and samples of this measure almost surely have no infinite cluster. If the free and wired random-cluster measures coincide, but the samples have an infinite cluster, then the Potts model itself is not \ffiid, but its gradient is. Finally, if the free and wired random-clusters do not coincide, then the gradient of Potts is not \ffiid.

We remark that the proof of \cref{thm:main_potts} can be easily adapted to show that the coding radius has exponential tails when $\beta$ is sufficiently large (as a function of $q$ and $d$). What is needed for this is that the corresponding random-cluster measure has a unique infinite cluster and satisfies~\eqref{eq:tree-exp-decay-cond}, something which can be shown to hold when $p$ is sufficiently close to 1.
We also remark that the part of \cref{thm:main_potts} showing that the gradient of Potts is not \ffiid, shows that it is in fact not $\Gamma$-\ffiid\ for any transitive subgroup $\Gamma$.

Finally, we mention that \cref{thm:main_potts} extends from the case of $\Z^d$ to any transitive locally-finite amenable graph $G$. Since the free and wired random-cluster measures with $q=2$ always coincide on such graphs~\cite{raoufi2017translation}, \cref{thm:main_Ising} also extends to this setting, except perhaps without the additional information on the coding radius.

\begin{proof}[Proof of \cref{thm:main_potts}]
Assume first that the free and wired random-cluster measures coincide and let $\omega$ be sampled from this measure. Recall that in the Edward--Sokal coupling (see, e.g., ~\cite[Theorem~4.91]{grimmett2006random}), given the percolation configuration $\omega$, to obtain a Potts model spin configuration $\sigma$ with constant 0 boundary conditions, we uniformly choose one of the $q$ colors independently for each finite cluster, and set the infinite cluster (if it exists) to have spin 0. By~\cite[Theorem~1.1]{harel2018finitary}, since the free and wired measures coincide, $\omega$ is \ffiid.
The well-known Burton--Keane argument implies that either $\omega$ almost surely has no infinite cluster or it almost surely has a unique infinite cluster.
In the former case (which can only occur when $p \le p_c$), it easily follows from the above description that $\sigma$ (and hence also $\nabla \sigma$) is \ffiid\ (as it is a finitary factor of $\omega$ and some independent \iid\ process). In the latter case, \cref{thm:grad_general} yields that $\nabla \sigma$ is \ffiid.

Assume now that the free and wired random-cluster measures are different. Let $\sigma$ be sampled from the constant 0 boundary condition Gibbs state for the Potts model and assume towards a contradiction that its gradient $\nabla \sigma$ is \ffiid. Recall that in the Edward--Sokal coupling, given the spin configuration $\sigma$, to obtain a sample $\omega$ from the wired random-cluster measure, we perform Bernoulli percolation with parameter $p=1-e^{-\beta}$ on the edges whose endpoints have equal spins in $\sigma$. Since the latter edges are obtained as a (local) function of the gradient $\nabla \sigma$, we see that $\omega$ is \ffiid\ (as it is a finitary factor of $\nabla \sigma$ and some independent \iid\ process).
However, by~\cite[Theorem~1.2]{harel2018finitary}, the wired (and free) random-cluster measure is not \ffiid\ whenever the free and wired measures are different. This leads to a contradiction, thus showing that the gradient of the Potts is not \ffiid.
\end{proof}

\begin{proof}[Proof of \cref{thm:main_Ising}]
Let $\omega$ be sampled from the unique FK-Ising ($q=2$) random-cluster measure with $p>p_c(d)$~\cite{B05}. As in the proof of \cref{thm:main_potts}, using the Edwards--Sokal coupling, we obtain an Ising spin configuration $\sigma$ (with the law of the plus state) by assigning an independent random sign to each finite cluster of $\omega$, and spin $+$ to the infinite cluster. By~\cite[Theorem~1.1]{harel2018finitary}, if the free and wired measures are exponentially close in the sense that
\[ \P^{\text{FK},\text{w}}_{B_n(\zero), p,q}(\omega_\zero=1) - \P^{\text{FK},\text{f}}_{B_n(\zero), p,q}(\omega_\zero=1) \le Ce^{-cn} ,\]
where $B_n(\zero)$ is the ball of radius $n$ around the origin, then $\omega$ is \ffiid\ with a coding radius having exponential tails. Thus, in light of \cref{thm:grad_general}, we need only check that this holds and that~\eqref{eq:tree-exp-decay-cond} holds. The former is shown in~\cite[Theorem~1.3]{duminil2018exponential} for $d \ge 3$ (for $d=2$ this is a simple consequence of planar duality and exponential decay in the subcritical regime~\cite{aizenman1987phase}). To show the latter, we rely on Pisztora's coarse graining approach.

Let $\phi = \P^{\text{FK},\text{w}}_{p,2} = \P^{\text{FK},\text{f}}_{p,2}$ denote the unique infinite-volume random-cluster measure and let $\phi_\Lambda^\xi$ denote the finite-volume measure in $\Lambda$ with boundary condition $\xi$. We recall the following notion of a \emph{good box} from~\cite{duminil2018exponential}. For $x \in \Z^d$, let $\Lambda_k(x)$ be the box around $x$ consisting of vertices at $\ell_\infty$-distance at most $k$ from $x$. Given $\omega$, we say a box $\Lambda_k$ is good if the following two conditions are satisfied:
\begin{enumerate}[{(}a{)}]
\item There exists an open cluster $B$ in $\Lambda_k$ touching all the $2d$ boundary faces of the box.
\item Any open path of length $k$ in $\Lambda_k$ belongs to $B$.
\end{enumerate}
The paper of Pisztora~\cite{pisz96} combined with that of Bodineau~\cite{B05} imply that there exists $c=c(p)>0$ such that for every $k$ and every boundary condition $\xi$,
$$
\phi_{\Lambda_{2k}(x)}^\xi (\Lambda_k(x)  \text{ is good}) \ge 1-e^{-ck}.
$$
(For ease of reference, let us point out that the above is statement (3.7) in Pisztora~\cite{pisz96}, but with $p > \hat p_1$ and $\alpha = 1$, where $\hat p_1$ is defined in (3.5) there, and it is proved in Bodineau~\cite{B05} that $\hat p_1 = p_c$.)

Now fix $\ve>0$ and $k  = k(\ve)$ so that the above event has probability at least $1-\ve$. Consider a site percolation $\eta=(\eta_x)_{x \in \Z^d}$ with a vertex $x$ open if the box $\Lambda_k(kx)$ is good and closed otherwise. It follows that, for any $x \in \Z^d$,
\[ \P(x\text{ is open in }\eta \mid (\eta_y)_{\|y-x\|_\infty \ge 3}) \ge 1-\ve \qquad\text{almost surely} .\]
It is a well-known result of Liggett, Schonmann, and Stacey~\cite[Theorem~0.0]{liggett1997domination} that this property implies that $\eta$ dominates a Bernoulli site percolation of density $1-\tilde \ve$ with $\tilde \ve \to 0$ as $\ve \to 0$. Thus, for small enough $\ve$, there is a unique infinite cluster $D$ in $\eta$, and the diameter of $K_\zero$ has exponential tails, where $K_\zero$ is the  connected component of $\Z^d \setminus D$ containing the origin. Indeed, if the diameter of $K_\zero$ is $n$, then there is a vertex $x$ at distance at most $n$ from $\zero$ which is part of a closed 2-connected component surrounding $\zero$ (blocking any path from $\zero$ to infinity) and thus of diameter at least $n$. A union bound gives the claimed exponential tail.

Let us now show that this yields \eqref{eq:tree-exp-decay-cond}.
By translation invariance of $\omega$, it suffices to prove~\eqref{eq:tree-exp-decay-cond} for the origin $\zero \in \Z^d$.
Let $H$ denote the set of all vertices at distance at most $10k$ from $\bigcup_{x \in K_\zero} \Lambda_k(kx)$.
Note that $\diam H$ is at most a constant (depending on $d$ and $k$) times $\diam K_0+1$, and hence has exponential tails.
Note that if $C_\zero$ is finite and has diameter larger than $k$, then $C_\zero \subset H$ by the properties defining a good box. This shows that $\P(r \le \diam C_\zero < \infty)$ decays exponentially.
Similarly, $\dist(C_\zero,C_\infty)$ is at most $\diam H+k$, so that it also has exponential tails. Finally, for $r$ large enough, $L_{\zero,r} \ge 4r$ implies that there is no surface of good boxes surrounding the origin consisting of boxes $\Lambda_k(kx)$ contained in the annulus $B_{3r}(\zero) \setminus B_{2r}(\zero)$. A union bound yields that this has probability exponentially  small in $r$. This establishes \eqref{eq:tree-exp-decay-cond}.
\end{proof}

\section{The beach model}
\label{sec:beach-model}

We consider here the multi-type beach model with $q \in \{2,3,\dots\}$ types and fugacity $\lambda>0$. The two-type beach model with integer fugacity was first introduced by Burton and Steif~\cite{burton1994non} (in the context of subshifts of finite type) and later extended to multiple types and real activities by Burton, Steif, H\"aggstr\"om and Hallberg~\cite{burton1995new,haggstrom1996phase,haggstrom2000markov,hallberg2004gibbs}.
In the beach model, each site $v$ is assigned a spin $\sigma_v=(\sigma^{\text{s}}_v,\sigma^{\text{t}}_v)$ consisting of a \emph{state} $\sigma^{\text{s}}_v \in \{0,1\}$ and a \emph{type} $\sigma^{\text{t}}_v \in \{0,\dots,q-1\}$.
Thus, a configuration in the beach model is an element $\sigma$ of $(\{0,1\} \times \{0,\dots,q-1\})^{\Z^d}$. 
Such a configuration is \textit{admissible} if any two neighboring spins are either of the same type or are both in a closed state, i.e., if $\sigma^{\text{t}}_u=\sigma^{\text{t}}_v$ or $\sigma^{\text{s}}_u=\sigma^{\text{s}}_v=0$ for any adjacent $u$ and~$v$.
Given a finite set $V \subset \Z^d$ and an admissible configuration $\tau$, the finite-volume Gibbs measure in $V$ with boundary condition $\tau$ is the probability measure $P^\tau_V$, which is supported on admissible configurations $\sigma$ that agree with $\tau$ outside $V$, and satisfying that, for every such $\sigma$,
\begin{equation}\label{eq:beach-spec}
P^\tau_V(\sigma) = \frac{1}{Z^\tau_V} \cdot \lambda^{\sum_{v \in V} \sigma_v^{\text{s}}} \cdot \1_{\{\sigma\text{ admissible}\}}  .
\end{equation}
Here $Z^\tau_V$ (which also depends on $q$ and $\lambda$) is a normalization constant.
A \emph{Gibbs measure} for the beach model is a probability measure $\mu$ on $(\{0,1\} \times \{0,\dots,q-1\})^{\Z^d}$, which is supported on admissible configurations, and such that a random configuration $\sigma$ distributed according to $\mu$ has the property that, for any finite $V \subset \Z^d$, conditioned on the restriction $\sigma|_{V^c}$, $\sigma$ is almost surely distributed according to $P^\sigma_V$.

There is a strong analogy between the multi-type beach model and the Potts model (and similarly between the two-type beach model and the Ising model). For instance, it is known~\cite{haggstrom1996phase,haggstrom1998random,hallberg2004gibbs} that there is a critical fugacity $\lambda_c(d) \in (0,\infty)$ such that there is a unique Gibbs measure at fugacity $\lambda<\lambda_c(d)$ and multiple such Gibbs measures at fugacity $\lambda>\lambda_c(d)$. Moreover, there are at least $q$ extremal Gibbs measures, one for each type, and these measures coincide if and only if there is a unique Gibbs measure.
These measures, which we call the constant-type Gibbs measures, are
obtained as limits of $P^\tau_V$ as $V$ increases to $\Z^d$ with the boundary condition $\tau$ in which all states are $1$ and all types identical. In particular, any two such measures are related to one another by a permutation of the types.
These results are a consequence of the existence of a random-cluster representation for the beach model, introduced by H\"aggstr\"om~\cite{haggstrom1996random,haggstrom1998random}, which serves as a graphical representation for the beach model much like the usual random-cluster model does for the Potts model. This beach-random-cluster model is very similar to the usual random-cluster model (with the notable difference that it lives on sites, not on edges). In particular, it is monotone and thus admits two extremal measures, which we call the free and wired beach-random-cluster measures.

It has been shown that the two-type beach model (more precisely, any constant-type Gibbs measure) is \ffiid\ if and only if there is a unique Gibbs measure (see~\cite[Corollary~1.7]{spinka2018finitarymrf}; the statement there refers to whether $\lambda$ is above or below $\lambda_c$, but the proof only relies on whether the Gibbs measure is unique or not), and the beach-random-cluster representation allows to extend this to the multi-type model. We are therefore led to consider a gradient of the model. The gradient we consider applies only to the types of the spins, leaving the information of their states intact. Precisely, the gradient of the types is the configuration on the oriented edges of $\Z^d$ given by
\[ (\nabla \sigma^{\text{t}})_e = \sigma^{\text{t}}_v - \sigma^{\text{t}}_u \mod q \]
for an oriented edge $(u,v) \in \vec E(\Z^d)$.
The gradient of $\sigma$ is then defined as the pair $\nabla\sigma=(\sigma^{\text{s}},\nabla \sigma^{\text{t}})$.

Continuing the analogy with the Potts model, we prove that the gradient of the beach model is \ffiid\ precisely when the corresponding beach-random-cluster model has a unique Gibbs state.

\begin{thm}\label{thm:main_beach}
Let $d \ge 2$ and $q \ge 2$ be integers and let $\lambda>0$. Let $\sigma$ be sampled from a constant-type Gibbs measure for the $q$-type beach model at fugacity~$\lambda$. The gradient $\nabla\sigma$ is \ffiid\ if and only if the free and wired measures of the associated beach-random-cluster model coincide.
\end{thm}

\begin{proof}
The proof is analogous to that of \cref{thm:main_potts}, with the beach-random-cluster model taking the place of the usual random-cluster model. We do not define this model here and refer to~\cite[Chapter~8]{hallberg2004gibbs} for definitions and results. We only mention that if $\sigma$ is sampled from a constant-type Gibbs measure for the beach model, then $\sigma^s$ has the law of the associated wired beach-random-cluster model.

Assume that the associated beach-random-cluster measure has a unique Gibbs measure (i.e., the free and wired measures coincide), and let $\omega \in \{0,1\}^{\Z^d}$ be sampled from this measure. We note that, unlike the usual random-cluster model, $\omega$ here lives on the sites of $\Z^d$. Since the beach-random-cluster model is monotone, \cite[Theorem~2.1]{harel2018finitary} (which roughly says that a monotone model whose extremal measures coincide is \ffiid) implies that $\omega$ is \ffiid. The Edwards--Sokal-like coupling between the beach model and the beach-random-cluster model (see~\cite[Proposition~8.8]{hallberg2004gibbs}) implies that a sample $\sigma$ from the constant-type-0 Gibbs measure can be obtained from $\omega$ by taking the states of $\sigma$ to be $\sigma^{\text{s}}=\omega$ and choosing the types of $\sigma$ randomly as follows: First let $\omega'$ be the edge percolation configuration in which an edge is open in $\omega'$ if and only if at least one of its endpoints is open in $\omega$, and then assign an independent uniform type in $\{0,\dots,q-1\}$ to each finite cluster of $\omega'$, and type $0$ to the infinite clusters.
Either $\omega'$ almost surely has no infinite cluster, in which case it follows that $\sigma$ itself is \ffiid, or $\omega'$ almost surely has a unique infinite cluster (by the Burton--Keane argument and since $\omega$ has finite energy), in which case it follows from \cref{thm:grad_general} that $\nabla \sigma$ is \ffiid.

Assume now that the free and wired beach-random-cluster measures are different.
Let $\sigma$ be sampled from the constant-type-0 Gibbs measure. Since $\sigma^s$ has the law of the wired beach-random-cluster measure, it suffices to show that the latter is not \ffiid. Indeed, the proof of this for the usual random-cluster model~\cite[Theorem~1.2]{harel2018finitary} easily extends to the beach-random-cluster model (the proof is based on ideas from~\cite{van1999existence}).
We conclude that $\nabla\sigma$ is not \ffiid.
\end{proof}

\section{The six-vertex model}
\label{sec:6v}

\begin{figure}
\centering
\includegraphics[scale = 1.4]{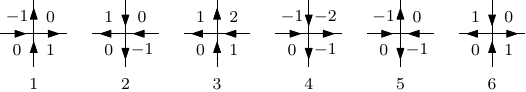}
\caption{The six types of arrow configurations satisfying the ice rule at a vertex, and the corresponding height function (which is assumed in the figure to be 0 on the bottom-left face).}\label{fig:ice}
\end{figure}

The six-vertex model is a model of arrow configurations on the edges of $\Z^2$ satisfying the \emph{ice rule}: at each vertex there are exactly two outgoing and two incoming arrows. This gives rise to one of six configurations (called \emph{types}) at each vertex, as depicted in \cref{fig:ice}. In the general setting, the six-vertex model assigns a different weight to each of the six types of vertices. In the special case considered here --- the so-called \emph{F-model} --- types 1 to 4 have weight 1 and types 5 and 6 have weight $c>0$. Roughly speaking, a six-vertex configuration is then randomly chosen with probability proportional to $c^{\#\{\text{type 5 or 6 vertices}\}}$.

The six-vertex model has an integer-valued height function representation. The relation between the six-vertex configuration and the height function is that the arrows of the former represent the gradient of the latter via the following convention: crossing an arrow from its left to its right increases the height by 1 (see \cref{fig:ice}). In fact, this mapping defines a bijection between six-vertex configurations and height functions modulo a global addition of an integer. We fix an additional convention that height functions take even values on the even sublattice, so that a six-vertex configuration determines the height function up to an addition of an even integer.

Recasting the model in terms of the height function representation, roughly speaking, a height function $h$ is randomly chosen with probability proportional to $c^{\#\text{saddle}(h)}$, where a \emph{saddle point} is a vertex of $\Z^2$ for which both diagonals have constant height. Indeed, saddle points of~$h$ correspond to vertices of type 5 and 6 in the six-vertex configuration (see \cref{fig:ice}).
We note that, unlike the Potts model, the six-vertex model has hard constraints and even admits \emph{frozen configurations} where no finite portion of the configuration can be modified in such a way that it still satisfies the ice rule (consider, for example, the arrow configuration in which every vertex has type~1, or equivalently, the height function given by $h(x ,y)=x-y$). In particular, the six-vertex model always has multiple (frozen) Gibbs states. Our results concern certain (non-frozen) Gibbs states, which we now define.

Let us proceed to give precise definitions.
We write $(\Z^2)^*$ for the dual lattice of $\Z^2$.
We write $\L$ and $\L^*$ for the even and odd sublattices of $(\Z^2)^*$, respectively, noting that each is a rotated and scaled copy of the integer lattice and that they are duals of each other. More precisely $\L := \sqrt{2}e^{i \pi/4} (\Z^2+(1/2,1/2))$ and $\L^*$ is its dual. A \emph{height function} is a function $h \colon (\Z^2)^* \to \Z$ such that $|h(u)-h(v)|=1$ for adjacent $u,v \in (\Z^2)^*$ and such that $h(u)$ is even for $u \in \L$ (and hence odd for $u \in \L^*$). We sometimes call $\L$ the even or primal sublattice, and $\L^*$ the odd or dual sublattice, depending on the context.

A \emph{diamond domain} is a set of the form $\Lambda = \{ u \in (\Z^2)^* : \dist(u,v) \le n \}$ for some $v \in \L$ and positive even integer $n$, where dist is the graph distance in $(\Z^2)^*$. The inner and outer vertex boundaries of such a diamond domain are $\{ u : \dist(u,v)=n \} \subset \L$ and $\{ u : \dist(u,v)=n+1 \} \subset \L^*$, respectively.
Though one could work with more general domains (so-called even domains), we stick to diamond domains for the sake of concreteness and clarity.
Given such a diamond domain and an even (resp. odd) integer~$m$, let $HF_\Lambda^m$ denote the set of height functions $h$ which equal $m$ (resp. $m+1$) on all the inner boundary of $\Lambda$ and $m+1$ (resp. $m$) on all the outer boundary of $\Lambda$, and which continue this pattern everywhere outside of $\Lambda$\footnote{This is just a convention; any other arbitrary but fixed assignment of heights would do just as well.}. We call this the \emph{$m$ boundary condition}. See \cref{fig:height}. Note that this definition ensures that (both inner and outer) boundary vertices take values in $\{m,m+1\}$, with the precise value determined according to the sublattice (even on $\L$, odd on $\L^*$). Thus, our boundary conditions are determined by an unordered pair of consecutive integers, and we chose to index these according to the smaller of the two integers.
Define a probability measure $\P^{\textsf{hf},m}_{\Lambda,c}$ on height functions by
\begin{equation}
\P^{\textsf{hf},m}_{\Lambda,c}(h) = \frac1{Z^{\textsf{hf},m}_{\Lambda,c}} \cdot c^{\# \text{saddle}(h)} \cdot \1_{HF_\Lambda^m}(h), \label{eq:hf}
\end{equation}
where $\# \text{saddle}(h)$ counts the number of saddle points of $h$ incident to a vertex in $\Lambda$ and $Z^{\textsf{hf},m}_{\Lambda,c}$ is the partition function.

\begin{figure}
\centering
\includegraphics[scale = 0.7]{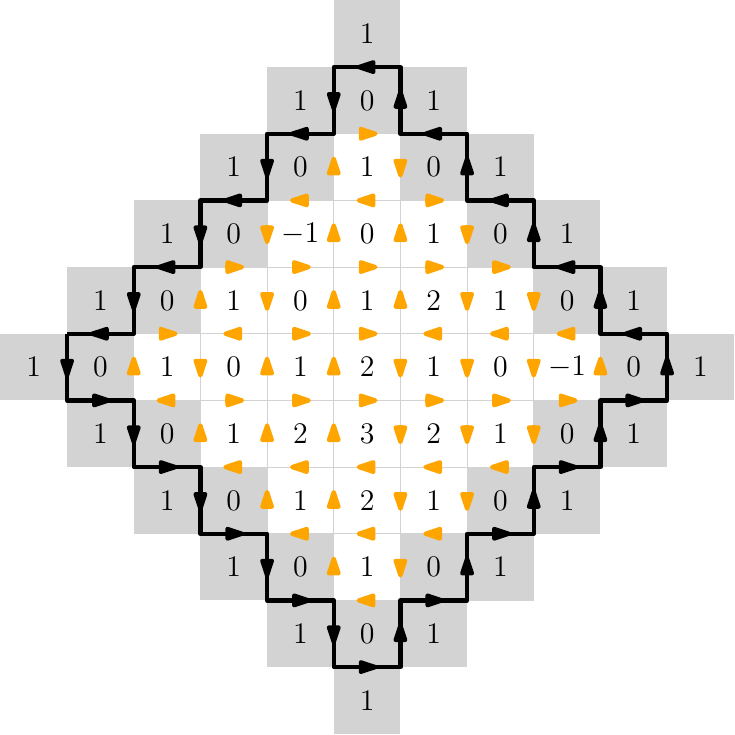}
\caption{A height function and its corresponding six-vertex configuration on a diamond domain $\Lambda$ with 0 boundary condition (meaning that the height is fixed to be 0 and 1 on the internal and external vertex boundaries). The black circuit is $\partial^\dagger \Lambda$.}
\label{fig:height}
\end{figure}

The \emph{gradient of a height function} $h$ lives on the oriented edges of $(\Z^2)^*$ and is defined by
\begin{equation}
 (\nabla h)_{(u,v)} = h(u) - h(v)  \label{eq:grad_h}
\end{equation}
for an oriented edge $(u,v) \in \vec E((\Z^2)^*)$.
We note that $\nabla h$ can be thought of as a six-vertex configuration, and that this correspondence between gradients of height functions and six-vertex configurations is a bijection.
We also define the \emph{diagonal gradient} to be the function on the oriented edges of $\L$ and $\L^*$ defined by
\begin{equation}
(\nabla_d h)_{(u,v)} = h(u) - h(v)  \label{eq:grad_d_h}
\end{equation}
 for an oriented edge $(u,v) \in \vec E(\L) \cup \vec E(\L^*)$. We note that $\nabla h$ and $\nabla_d h$ are defined by the same formula, but on different domains. We write $|\nabla_d h|$ for the pointwise absolute value of $\nabla_d h$ and note that $|\nabla_d h|_{(u,v)} = 2 \cdot \1_{\{h(u) \neq h(v) \}}$ so that $|\nabla_d h|$ may be thought of as a function on the non-directed edges $E(\L) \cup E(\L^*)$. Note that the pointwise absolute value of $\nabla h$ is not an interesting object as it is always the constant 1 function.
Let us consider yet another object of interest. Define the \emph{Laplacian} of $h$ (or the ``curl" of the six-vertex configuration) to be the function from $(\Z^2)^*$ to $\{0,\pm \frac12, \pm 1\}$ given by
\begin{equation}
(\Delta h)_u = \frac14 \sum_{v \sim u} (h(v)-h(u)), \label{eq:eq:lap_h}
\end{equation}
where the sum is over the four neighbors of $u$ in $(\Z^2)^*$.
We also write $|\Delta h|$ for the pointwise absolute value of $\Delta h$.

We note that measures on height functions cannot be \ffiid\ for the trivial reason that height functions are always even on $\L$ and odd on $\L^*$. It is therefore natural to ask instead whether they are $(\Z^2)_\text{even}$-\ffiid, where $(\Z^2)_\text{even}$ is the group of translations which preserve the two sublattices. On the other hand, the gradient and Laplacian do not suffer from this problem, and could potentially have a coding which commutes with all automorphisms.

\begin{thm}\label{thm:h_grad}
Let $p_c$ denote the critical probability for Bernoulli site percolation on $\Z^2$, and fix $c>\frac{2+p_c}{1-p_c}$ and $m \in \Z$. Then $\P^{\textsf{hf},m}_{\Lambda,c}$ converges to an infinite-volume limit $\P^{\textsf{hf},m}_c$ as $\Lambda$ increases to $(\Z^2)^*$ along diamond domains. Moreover, if $h$ is sampled from $\P^{\textsf{hf},m}_c$, then
\begin{enumerate}
 \item The random fields $h$, $\nabla h$, $\nabla_d h$ and $\Delta h$ are not $(\Z^2)_{\text{even}}$-\ffiid.
 \item The random fields $|\nabla_d h|$ and $|\Delta h|$ are \ffiid.
\end{enumerate}
\end{thm}

The theorem shows that while the Laplacian and diagonal gradients of the height function are not $(\Z^2)_{\text{even}}$-\ffiid, their absolute values are.
Let us already mention here that $\Delta h$ is a simple local function of $\nabla h$, which in turn is a simple (non-local, but finitary) function of $\nabla_d h$. Thus, the first part of \cref{thm:h_grad} boils down to showing that $\Delta h$ is not $(\Z^2)_{\text{even}}$-\ffiid. Similarly, $|\Delta h|$ is a simple local function of $|\nabla_d h|$ so that the second part of \cref{thm:h_grad} boils down to showing that $|\nabla_d h|$ is \ffiid.

We mention that Glazman and Peled~\cite{glazmanpeled2019} showed that $\P^{\textsf{hf},m}_{\Lambda,c}$ converges for all $c>2$. We give a different and self-contained proof of this for $c>\frac{2+p_c}{1-p_c}$.
For $c=2$, these measures do not converge (the height function has logarithmic variance), though the gradient measures do (that is, the six-vertex measures converge) and the limiting measure does not depend on $m$~\cite{glazmanpeled2019}. In this case, it can be shown that this measure is \ffiid\ (see \cref{rem:c=2}).

\medbreak

Let us now outline some of the ideas and ingredients that go into the proof \cref{thm:h_grad}.
The part of the result concerning the non-existence of a finitary coding follows a similar argument as the one in~\cite{spinka2018finitarymrf} for general Markov random fields (though the argument here is slightly complicated by the existence of hard constraints) and we do not expand on it here.
Let us explain the second part of the result, namely, that the absolute value of the diagonal gradient is \ffiid.

A significant difference between the six-vertex model height function and the Potts (or beach) model is that the state space is not compact. A first step toward incorporating this model into the framework of \cref{sec:general} is to find a spin representation. It turns out this can essentially be done by taking the values modulo $4$ of the heights and using the bipartiteness of the square lattice to reduce it to a spin model with two spin values. This operation preserves the absolute value of the diagonal gradient of the height function (it is zero precisely when the spins are equal), so that the goal becomes to show that the diagonal gradient of the spins is \ffiid.
The spin representation is defined in \cref{sec:spin}.

The next key ingredient is a new percolation model which we call the \emph{superimposed random-cluster model} (or just the \emph{superimposed model} for short). As its name suggests, the superimposed random-cluster model consists of two random-cluster models, one on the primal lattice $\L$ and one on the dual lattice $\L^*$, superimposed on top of each other. Roughly speaking, the two random-cluster models are sampled independently of each other and then conditioned to have no \emph{closed crosses}. Precise definitions are given in \cref{sec:si}.

The superimposed model serves as a graphical representation of the spin representation of the six-vertex model with $c \ge 2$, much like the usual random-cluster model serves as a graphical representation of the Ising and Potts models (e.g., the spin representation of the six-vertex model can be coupled with the superimposed model in a manner reminiscent of the usual Edwards--Sokal coupling, where an open edge forces its two endpoints to have equal spins); see \cref{fig:si}.
We therefore believe that this model is of independent interest.

\begin{figure}
\centering
\includegraphics[scale = 0.6]{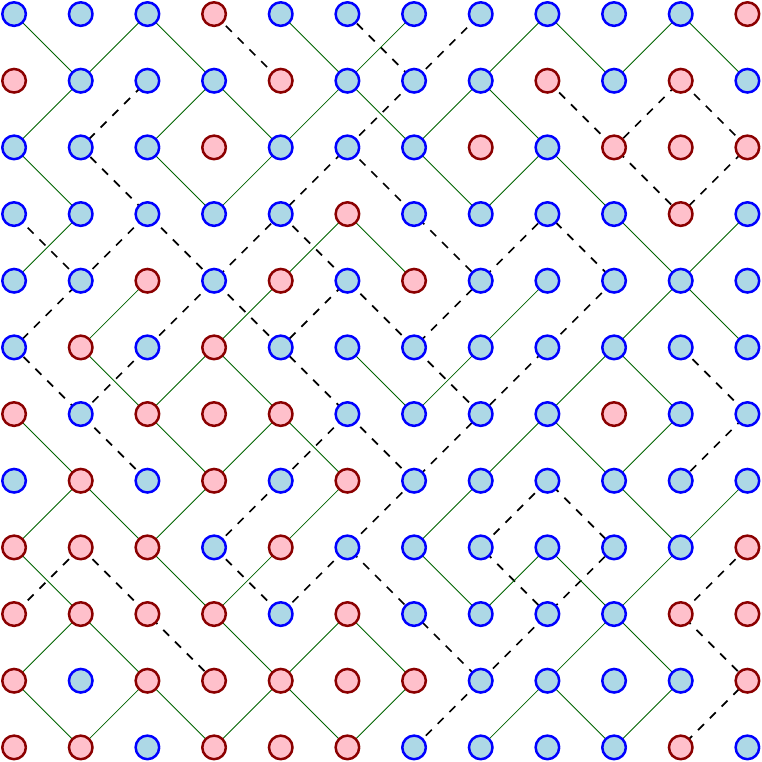}
\caption{A superimposed configuration $\eta$ with a compatible spin configuration $\sigma$. The blue and red circles correspond to $-$ and $+$ spins in $\sigma$, respectively. Sites which are connected by edges of $\eta$ must have the same spin value in $\sigma$. Unlike in the usual random-cluster model, an edge and its dual can both be present in $\eta$.}\label{fig:si}
\end{figure}

To emphasize this last point, we mention here that graphical representations of lattice spin models have gained immense popularity in recent times. Besides the random-cluster model, examples include a whole range of very popular models such as the random current model, the high and low temperature expansions of the Ising model, cluster expansions, random walk representations, and the loop $O(n)$ model. We refer to \cite{duminil_graphical,peled2017lectures} for excellent surveys on this subject. Such representations translate information about correlations in the spin model into connectivity properties of a percolation-type model arising from the graphical representation and have been used as a central tool in settling various open problems \cite{aizenman2015random,duminil2017continuity,aizenman2018emergent,duminil2018exponential}. We mention that the (critical) random-cluster model is itself related to the six-vertex model via the Baxter--Kelland--Wu coupling (see \cite{BKW76}), but that this is not analogous to the relation between the Ising/Potts model and the random-cluster model (see \cref{rem:bkw}).
Motivated to find such an analogue, we discovered the aforementioned superimposed model.

In \cref{sec:si-spin-coupling} we establish the above Edwards--Sokal-like coupling in finite domains, and then continue to establish some other crucial properties of the superimposed model in the following subsections. Specifically, in \cref{sec:monotone} we show that the model satisfies a monotonicity (FKG) property (albeit with a partial order which is reversed on one sublattice). We then show in \cref{sec:uniqueness} that, in a certain regime of its parameters, the model has a unique Gibbs state and that samples from this unique Gibbs state have a unique infinite cluster in each sublattice. In \cref{sec:inf_coupling} we extend the above coupling to the infinite volume.
Once we are equipped with these properties of the superimposed model, the proof of \cref{thm:h_grad} is similar to that of \cref{thm:main_potts}.
Indeed, the general result in~\cite{harel2018finitary} (which relies on the monotonicity and uniqueness of the Gibbs measure) will imply that the superimposed model is \ffiid, and then the general result shown in \cref{sec:general} (which relies on the uniqueness of the infinite cluster) will imply that the diagonal gradient of the spin representation (and hence also $|\nabla_d h|$) is \ffiid. This details of this and related statements are given in \cref{sec:h_grad}.

We end with some notation which will be used throughout the section.
Let $\Lambda$ be a diamond domain and recall that this is a subset of $(\Z^2)^*$. Let $\partial^\dagger \Lambda$ be the simple circuit in $\Z^2$ which lies between the inner and the outer boundary vertices of~$\Lambda$ (see \cref{fig:height}).
Let $\hat \Lambda$ denote the subgraph of $\Z^2$ induced by the vertices of $\partial ^\dagger \Lambda$ and all the vertices of $\Z^2$ enclosed by it. \emph{Internal vertices} of $\hat \Lambda$ are those vertices of $\hat \Lambda$ which have all four incident edges belonging to $\hat \Lambda$ (these include all vertices enclosed by $\partial ^\dagger \Lambda$, and also some vertices of $\partial ^\dagger \Lambda$).

\subsection{The spin representation}\label{sec:spin}

The spin representation of the six-vertex model is a spin model on $(\Z^2)^*$.
A configuration in this model is an element $\sigma \in \{-,+\}^{(\Z^2)^*}$ which satisfies the \emph{ice rule}:
in any $2\times2$ square, at least one of the two diagonals consists of equal spins.
Every height function $h$ projects onto such a spin configuration $\sigma$ given by
\begin{equation}\label{eq:height-to-spin}
\sigma_v = 
\begin{cases}
+ &\text{if }h(v) = 0 ,1 \mod 4\\
- &\text{if }h(v) = 2 ,3 \mod 4
\end{cases}.
\end{equation}
In the other direction, every spin configuration lifts to countably many height functions which differ from one another by a global additional of an integer in $4\Z$ (recall that by definition, we force the height function to be even on $\L$).
In particular, any six-vertex configuration lifts to precisely two spin configurations, which are global flips of each other (i.e., one is $\sigma$ and the other is $-\sigma$). See \cref{fig:spin}.

\begin{figure}[t]
\centering
\includegraphics[scale = 1.4]{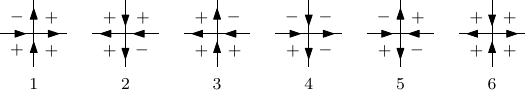}
\caption{The spin representation of the six-vertex model.  The spin is $+$ if the height is 0 or 1 modulo 4, and $-$ otherwise. The height is assumed in the figure to be 0 on the bottom-left square. If instead it were assumed to be 2, then the spins would be globally flipped. If it were 1 or 3, the same rule would apply and would result in a similar figure.} \label{fig:spin}
\end{figure}

Recall the definitions of the gradient of a height function \eqref{eq:grad_h} and its Laplacian \eqref{eq:eq:lap_h}. We make the following straightforward observations that if a height function $h$ projects to the spin configuration $\sigma$, then, for $v \in \L$,
\begin{align}
(\nabla h)_{(u,v)} &= \1_{\{\sigma_u=\sigma_v\}} - \1_{\{\sigma_u \neq \sigma_v\}}, \label{eq:spin-height-nabla}\\
(\Delta h)_v &= \frac14 \sum_{u \sim v} (\1_{\{\sigma_u=\sigma_v\}} - \1_{\{\sigma_u \neq \sigma_v\}}) , \label{eq:spin-height-delta}\\
|(\Delta h)_v| &= \frac14 \Big| \#\{ u \sim v : \sigma_u = +\} - \#\{ u \sim v : \sigma_u = -\} \Big| . \label{eq:spin-height-abs-delta}
\end{align}

By pushing $\P^{\textsf{hf},m}_{\Lambda,c}$ forward via the projection from height functions to spin configurations, one obtains a corresponding measure on spin configurations. Note that, by~\eqref{eq:height-to-spin} and the convention regarding the height function boundary condition, this measure is supported on spin configurations whose inner boundary vertices have spin $+$ if and only if $m=0,3$ mod~4, and whose outer boundary vertices have spin $+$ if and only if $m=0,1$ mod~4. Moreover, any two values of $m$ which are congruent mod 4 induce the same measure on spin configurations, so that only four such measures arise in this manner. We denote these measures by $\{\P^{\textsf{spin},ij}_{\Lambda,c}\}_{i,j \in \{+,-\}}$, where $\P^{\textsf{spin},ij}_{\Lambda,c}$ corresponds to the case where the outer boundary has spin $i$ and the inner boundary has spin $j$. It is easy to check that these measures are given explicitly by
\begin{equation}
\P^{\textsf{spin},ij}_{\Lambda,c}(\sigma) = \frac1{Z^{\textsf{spin},ij}_{\Lambda,c}} c^{\#\text{saddle}_\Lambda(\sigma)}\1_{\Omega^{\mathsf{spin},ij}_\Lambda}(\sigma).\label{eq:spin_boundary}
\end{equation}
Here $\Omega^{\mathsf{spin},ij}_{\Lambda}$ is the space of all spin configurations $\sigma \in \{+,-\}^{(\Z^2)^*}$ satisfying the ice rule, having spin $i$ on outer boundary vertices, spin $j$ on inner boundary vertices and continuing this pattern everywhere outside of $\Lambda$ (i.e., $i$ on even, $j$ on odd), and saddle$_\Lambda(\sigma)$ is the set of internal vertices of $\hat\Lambda$ which have type 5 or 6 (see \cref{fig:spin}), and $Z^{\textsf{spin},ij}_{\Lambda,c}$ is the appropriate partition function. We refer to type 5 or 6 vertices of $\Z^2$ as \emph{saddle points} from now on.

\subsection{The superimposed model} \label{sec:si}

In this section, we define the superimposed model. Though for our application we will use this model with parameter $q=2$, we introduce the model with general $q>0$ as it may be of independent interest.

The superimposed model consists of two random-cluster configurations, one on the primal lattice $\L$ (which is the rotated and scaled copy of $\Z^2$ formed by the even vertices of $(\Z^2)^*$) and one on its dual lattice $\L^*$. We think of $\L$ and $\L^*$ as the graphs (isomorphic to the square lattice) induced by their vertices, with $E(\L)$ and $E(\L^*)$ denoting their edge sets.
A \emph{cross} is a pair $\{e,e^*\}$ of primal/dual edges, where $e \in \L$ and $e^* \in \L^*$ is its dual edge.
Configurations of the superimposed model are pairs $\eta=(\eta^0,\eta^1) \in \{0,1\}^{E(\mathbb L)} \times \{0,1\}^{E(\mathbb L^*)}$ which contain no closed crosses, where a cross $\{e,e^*\}$ is said to be \emph{closed} if $\eta^0_e=\eta^1_{e^*}=0$, and \emph{open} if $\eta^0_e=\eta^1_{e^*}=1$.
We may also regard $\eta$ as an element of $\{0,1\}^{E(\L) \cup E(\L^*)}$.
Note that the set of crosses may be identified with $\Z^2$ (the intersection point of $e$ and $e^*$ lies on a vertex of the $\Z^2$ lattice). Thus, we may also identify $\eta$ with a configuration $\hat\eta \in \{(1,0),(1,1),(0,1)\}^{\Z^2}$ (i.e. a three-state site percolation on $\Z^2$), where the three values correspond to the possible values of $(\eta^0,\eta^1)$ on any cross.

Let $\Omega^{\mathsf{SI}}$ be the set of superimposed configurations.
The superimposed measure with parameters $\alpha>0$ and $q>0$ and boundary condition $\tau \in \Omega^{\mathsf{SI}}$ on a finite set $\Delta \subset E(\L) \cup E(\L^*)$ is given by
\begin{equation}
\P_{\Delta, \alpha,q}^{\mathsf{SI},\tau} (\eta) = \frac1 { Z_{\Delta, \alpha,q}^{\mathsf{SI},\tau}}\alpha^{N_\Delta(\eta)} q^{k_\Delta(\eta)} \1_{\Omega^{\mathsf{SI}, \tau}_\Delta}(\eta), \label{eq:superimposed}
\end{equation}
where $N_\Delta(\eta)$ is the number of open crosses of $\eta$ intersecting $\Delta$ (i.e., open crosses $\{e,e^*\}$ such that $e \in \Delta$ or $e^* \in \Delta$ or both), $k_\Delta(\eta)$ is the sum $k_\Delta(\eta^0)+k_\Delta(\eta^1)$ of the number of open vertex-clusters in $\eta^0$ and $\eta^1$ that contain a vertex incident to $\Delta$, and $\Omega^{\mathsf{SI},\tau}_\Delta$ is the set of configurations in $\Omega^{\mathsf{SI}}$ which agree with $\tau$ outside $\Delta$.

Of particular interest will be the boundary condition $\tau$ in which all edges are open. We call this boundary condition the \emph{wired-wired} boundary condition (since we are wiring both primal and dual edges) and denote the corresponding measure by $\P^{\mathsf{SI},11}_{\Delta, \alpha,q}$.
Similarly, the \emph{wired-free} (resp.\ \emph{free-wired}) boundary condition is the configuration $\tau$ in which all primal edges are open (resp.\ closed) and all dual edges are closed (resp.\ open), and the corresponding measure is denoted by $\P^{\mathsf{SI},10}_{\Delta,\alpha,q}$ (resp.\ $\P^{\mathsf{SI},01}_{\Delta,\alpha,q}$).

Recall the definition of a diamond domain from \cref{sec:6v} (see \cref{fig:height}).
Given a diamond domain $\Lambda \subset (\Z^2)^*$, we define $\Delta=\Delta(\Lambda)$ to be the union of the set of edges of $\L$ having both endpoints in $\Lambda$, and their dual edges (note that these are precisely the edges of $\L^*$ having at least one endpoint in $\Lambda$).

\subsection{A graphical representation}
\label{sec:si-spin-coupling}
In this section, we define an Edwards--Sokal-like coupling between the spin representation of the six-vertex model and the superimposed model with $q=2$. We construct this coupling in finite domains here (more precisely, diamond domains for clarity) and later extend it to infinite volume in \cref{sec:inf_coupling}.

Recall the notation from \cref{sec:si}.
Let $\Lambda$ be a diamond domain and set $\Delta = \Delta(\Lambda)$.
Let $\sigma \in \Omega^{\mathsf{spin},ij}_{\Lambda}$ and $\eta \in \Omega^{\mathsf{SI},11}_\Delta$.
 We say that $\sigma$ and $\eta$ are \emph{compatible}, denoted by $\sigma \sim \eta$, if
\[ \eta_e=1 ~\implies~ \sigma_u = \sigma_v \qquad\text{for every } e=\{u,v\} \in E(\L) \cup E(\L^*) .\]
Define a probability measure on $\Omega^{\mathsf{spin},ij}_{\Lambda} \times \Omega^{\mathsf{SI},11}_\Delta$ by
\begin{equation}
 \bP^{ij}_{\Lambda, \alpha}(\sigma, \eta) = \frac1{ \bZ^{ij}_{\Lambda, \alpha}}  \alpha^{N_\Delta(\eta)} \1_{\{\sigma \sim \eta\}} , \label{eq:ES_coupling}
\end{equation}
where $ \bZ^{ij}_{\Lambda, \alpha}$ is the appropriate partition function.

\begin{figure}[t]
\centering
\includegraphics[scale = 0.54]{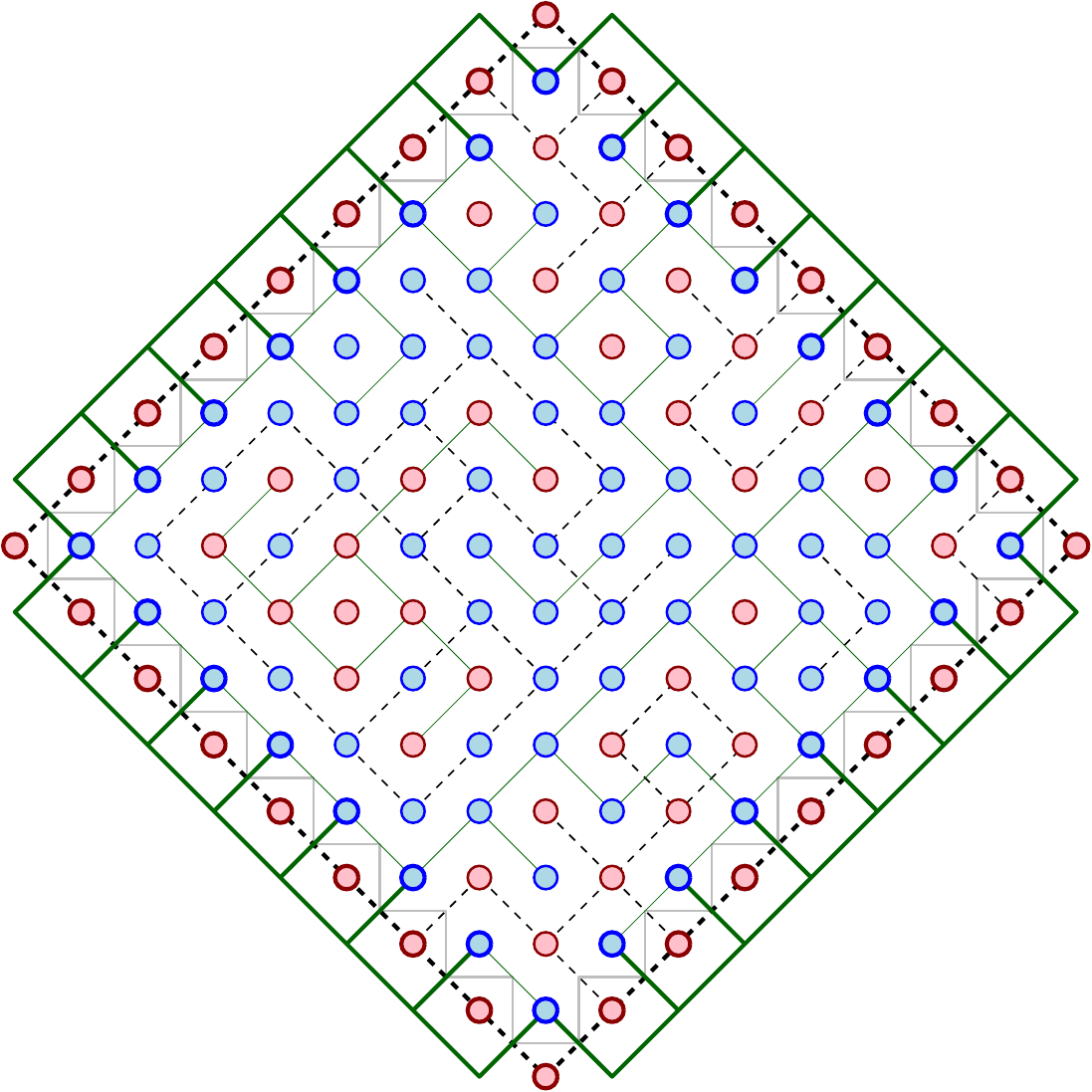}
\caption{A superimposed configuration with wired-wired boundary conditions and a compatible spin configuration with $+-$ boundary conditions ($+$ in red, $-$ in blue) on a diamond domain.}\label{fig:si-coupling}
\end{figure}

\begin{prop}\label{P:ES_coupling}
Let $c>2$ and set $\alpha=c-2$. Let $i,j \in \{+,-\}$ and let $\Lambda$ be a diamond domain and $\Delta = \Delta(\Lambda)$. Then $\bP^{ij}_{\Lambda,\alpha}$ defines a coupling between $\P^{\textsf{spin},ij}_{\Lambda,c}$ and $\P^{\mathsf{SI},11}_{\Delta,\alpha,2}$ . Moreover, if $(\sigma,\eta)$ is sampled from this coupling, then
\begin{itemize}
 \item Given $\sigma,$ $\eta$ can be sampled by first putting in the unique compatible edge at each non-saddle point, and then, independently for each saddle point, assigning one of the three values $(1,0),(1,1),(0,1)$ for $\hat\eta$ with probabilities $\frac1{2+\alpha}, \frac\alpha{2+\alpha},\frac1{2+\alpha}$, respectively.
 \item Given $\eta$, $\sigma$ can be sampled by independently choosing a uniform sign for each non-boundary cluster. The two boundary clusters receive the spin prescribed by the boundary condition (i.e. $i$ for the dual boundary cluster and $j$ for the primal boundary cluster).
\end{itemize}
\end{prop}

\begin{proof}
To compute the first marginal, we fix a spin configuration $\sigma \in \Omega^{\mathsf{spin},ij}_{\Lambda}$ and sum over $\eta$ to get
$$
 \sum_{\eta \in  \Omega^{\mathsf{SI},11}_\Delta}  \bP^{ij}_{\Lambda,\alpha}(\sigma,\eta)=
 \frac1{ \bZ^{ij}_{\Lambda, \alpha}} \sum_{\substack{\eta \in  \Omega^{\mathsf{SI},11}_\Delta \\\eta \sim \sigma}}
  \alpha^{N_\Delta(\eta)}
 =
  \frac1{ \bZ^{ij}_{\Lambda, \alpha}} (2+\alpha)^{\#\text{saddle}_\Lambda(\sigma)} .
 $$
Indeed, for each internal vertex in $\hat \Lambda$ which is not a saddle point, the edge of $\eta$ joining the diagonal with equal spins is forced to be open (due to the compatibility requirement between $\sigma$ and $\eta$), and contributes weight 1. For internal saddle points, either exactly one of the primal or dual edges could be open (each such possibility contributing weight 1) or both could be open, contributing weight $\alpha$. Thus, the overall contribution to the weight from each internal saddle point is $2+\alpha$. Finally, each non-internal vertex is necessarily a saddle point in $\sigma$ and is forced to be an open cross in $\eta$ by the wired-wired boundary conditions. We emphasize here that the boundary condition neither forbids nor forces the internal vertices of $\hat\Lambda$ in $\partial ^\dagger \Lambda$ to be saddle points. Since $\text{saddle}_\Lambda(\sigma)$ only consists of internal saddle points, this leads to the above equality.
Comparing this expression with~\eqref{eq:spin_boundary}, we see that the first marginal is exactly $\P^{\textsf{spin},ij}_{\Lambda,c}$ (in fact, we also see that $\bZ^{ij}_{\Lambda, \alpha} = Z^{\textsf{spin},ij}_{\Lambda,c}$).

To compute the second marginal, we fix a superimposed configuration $\eta \in \Omega^{\mathsf{SI},11}_\Delta$ and sum over $\sigma$ to obtain
$$ \sum_{\sigma  \in \Omega^{\mathsf{spin},ij}_{\Lambda}} \bP^{ij}_{\Lambda,\alpha}(\sigma,\eta) =
\frac1{ \bZ^{ij}_{\Lambda, \alpha}} \alpha^{N_\Delta(\eta)} |\{\sigma  \in \Omega^{\mathsf{spin},ij}_{\Lambda} : \sigma \sim \eta \}|
=
 \frac1{ \bZ^{ij}_{\Lambda, \alpha}} \alpha^{N_\Delta(\eta)} 2^{k_\Delta(\eta)-2} .
$$
Indeed, due to the compatibility requirement between $\sigma$ and $\eta$, each cluster in $\eta$ must receive a single spin (that is, all vertices in a given cluster must receive the same spin). Also we have two choices for this spin for each non-boundary cluster, and these choices can be made independently of each other.
However, the spins of the unique primal and dual boundary clusters (because of the wired boundary condition) are determined by the boundary condition.
Comparing this expression with~\eqref{eq:superimposed}, we see that the second marginal is exactly $\P^{\mathsf{SI},11}_{\Delta,\alpha,2}$ (in fact, we also see that $4\bZ^{ij}_{\Lambda,\alpha} = Z_{\Delta, \alpha,2}^{\mathsf{SI},\tau}$).

The description of the conditional laws is now immediate.
\end{proof}

\begin{remark}
The above coupling can be extended to more general domains and boundary conditions in the same spirit as for the Potts and random-cluster models.
\end{remark}

\begin{remark}\label{rem:bkw}
Let us mention a connection with the BKW coupling~\cite{BKW76}, for those familiar with it.
The BKW coupling is a coupling between the six-vertex model with $c \ge 2$ and the random-cluster model with $q \ge 4$.
\cref{P:ES_coupling} gives a coupling between the six-vertex model with $c \ge 2$ and the superimposed model with $q=2$ and $\alpha \ge 0$.
This gives a coupling $P$ of all three models together, where the random-cluster and superimposed models are conditionally independent given the spin representation.
When $\alpha=0$, open crosses are not allowed, and it can be checked that the superimposed model with $q=2$ coincides with the critical random-cluster model with $q=4$.
In this case, the coupling in \cref{P:ES_coupling} and the BKW coupling are essentially the same. However, the coupling induced by $P$ does not reflect this fact.
This raises the question of whether there is a more natural coupling between the superimposed model with $q=2$ and the random-cluster model with $q \ge 4$.
\end{remark}
\begin{remark}
The six-vertex model considered in this paper can be obtained as an infinite-coupling limit of the mixed Ashkin--Teller model in the sense of \cite{huang2013hintermann}. By a calculation inspired by the one presented in \cite{chayes2000lebowitz}, one can obtain the superimposed model as the limit of a random-cluster representation of the mixed Ashkin--Teller model. We also point out the paper \cite{pfister1997random} where a random-cluster representation of the Ashkin--Teller model was studied. We thank Alexander Glazman and Ron Peled for bringing these papers to our attention.
\end{remark}

\subsection{Monotonicity}
\label{sec:monotone}

The superimposed model possesses a monotonicity property (FKG) with respect to boundary conditions.
It is easy to see that the model is actually not monotonic in the usual partial order on both lattices. Nevertheless, it turns out that it is monotonic with respect to the partial order that is reversed on one of the lattices. Precisely, denote by $\preceq$ the partial order on $\{0,1\}^{E(\L) \cup E(\L^*)}$ obtained from the usual pointwise order on $E(\L)$ and the reverse order on $E(\L^*)$. That is, for $\eta,\xi \in \{0,1\}^{E(\L) \cup E(\L^*)}$,
\begin{equation}
\eta \preceq \xi \qquad\text{if and only if}\qquad \eta^0 \le \xi^0 \text{ and }\eta^1 \ge \xi^1, \label{eq:po}
\end{equation}
where $\le$ is used to denote the usual pointwise order.
Recall that $\eta$ may be viewed as an element $\hat\eta$ of $\{(1,0),(1,1),(0,1)\}^{\Z^2}$ according to the possible values of $(\eta^0,\eta^1)$ on any cross. We note that if one replaces the three values $(1,0),(1,1),(0,1)$ with $1,0,-1$, respectively, then the above partial order simply translates to the usual pointwise order on $\{1,0,-1\}^{\Z^2}$ in the sense that $\eta \preceq \xi$ if and only if $\hat\eta \le \hat\xi$.
 
\begin{prop}\label{prop:FKG}
Fix $q \ge 1$ and $\alpha>0$. Let $\Delta \subset E(\L) \cup E(\L^*)$ be finite and let $\tau, \tau' \in \Omega^{\mathsf{SI}}$ be two boundary conditions such that $\tau \preceq \tau'$. Then $ \P^{\mathsf{SI},\tau'}_{\Delta, \alpha,q}$ stochastically dominates $\P^{\mathsf{SI},\tau}_{\Delta, \alpha,q}$.
\end{prop}
\begin{proof}
By~\cite[Theorem~4.8]{georgii2001random}, we only need to check \emph{Holley's criterion}. That is, we need to check that, for any $e \in E(\L)$ and $\xi,\xi' \in \Omega^{\mathsf{SI}}$ with $\xi \preceq \xi'$,
\[ \P^{\mathsf{SI},\xi}_{\{e\},\alpha,q}(\eta_e=1) \le \P^{\mathsf{SI},\xi'}_{\{e\},\alpha,q}(\eta_e=1) \qquad\text{and}\qquad \P^{\mathsf{SI},\xi}_{\{e^*\} , \alpha,q} (\eta_{e^*}=1) \ge \P^{\mathsf{SI},\xi'}_{\{e^*\} , \alpha,q} (\eta_{e^*}=1) .\]
Note that we used here the domain Markov property, namely, that
\[ \P^{\mathsf{SI},\tau}_{\Delta,\alpha,q}(\eta_e \in \cdot \mid \eta = \xi\text{ on }\Delta \setminus \{e\})=\P^{\mathsf{SI},\xi}_{\{e\},\alpha,q} .\]
Now observe that by~\eqref{eq:superimposed}, for any $e=\{u,v\} \in E(\L) \cup E(\L^*)$,
\begin{equation}\label{eq:cond_superimposed}
\P^{\mathsf{SI},\xi}_{\{e\} , \alpha,q} (\eta_e=1) = \begin{cases}
 1 &\text{if } \xi_{e^*} = 0\\
 \frac{\alpha}{1+\alpha} &\text{if } \xi_{e^*} = 1,~ u \overset{\xi \setminus \{e\}}\leftrightarrow v\\
 \frac{\alpha}{q+\alpha} &\text{if } \xi_{e^*} = 1,~ u \overset{\xi \setminus \{e\}}{\not\leftrightarrow} v
\end{cases} .
\end{equation}
Finally, since $q \ge 1$, it is straightforward to verify Holley's criterion.
\end{proof}

\begin{remark}
As we mentioned before, flipping the order in one of the lattices is crucial for the stochastic domination to hold; a similar behavior exists in hardcore model~\cite{georgii2001random}.
Another interesting question concerns monotonicity of the model in the parameter $\alpha$. It is unclear whether the measures $\P^{\mathsf{SI},\tau}_{\Delta, \alpha,q}$ (or their marginals on $\eta^0$ and $\eta^1$) are monotonic in $\alpha$ (in the usual pointwise order). See \cref{quest:monotone}.
\end{remark}

Recall the wired-free and free-wired boundary conditions from \cref{sec:si}.
Note that the wired-free  and free-wired boundary conditions correspond to the unique maximal and minimal elements in $\Omega^{\mathsf {SI}}$ according to the above partial order. It immediately follows from \cref{prop:FKG} that the wired-free measure $\P^{\mathsf{SI},10}_{\Delta,\alpha,q}$ (resp. free-wired measure $\P^{\mathsf{SI},01}_{\Delta,\alpha,q}$) is the biggest (resp. smallest) superimposed measure in $\Delta$ in the sense of stochastic. In particular, the wired-wired measure $\P^{\mathsf{SI},11}_{\Delta,\alpha,q}$ (which played an important role in the coupling with the six-vertex model in \cref{sec:si-spin-coupling}) lies in between these two extremal measures.

\subsection{Uniqueness of the Gibbs measure for large $\alpha$}\label{sec:uniqueness}

The goal of this section is to establish the existence of a unique infinite-volume superimposed measure for large $\alpha$. For this argument, we do not require the monotonicity established in the previous section and hence it applies to all $q>0$. We prove uniqueness for large enough $\alpha$, and do not know whether it holds in general; see \cref{quest:uniqueness}.

The unique measure obtained will be translation-invariant on $\{0,1\}^{E(\L) \cup E(\L^*)}$. Let us first define precisely what we mean by this (as there are several lattices around). A translation $T \colon \Z^2 \to \Z^2$ can be viewed also as a translation on $E(\L) \cup E(\L^*)$ by $T(e) = \{T(u),T(v)\}$ for $e=\{u,v\} \in E(\L) \cup E(\L^*)$. 
A measure $\mu$ on $\{0,1\}^{E(\L) \cup E(\L^*)}$ is translation-invariant if it is preserved by any such translation.

We also give some consequences of monotonicity in the case $q \ge 1$. In this case, there are two extremal infinite-volume superimposed measures. Indeed, it follows from \cref{prop:FKG} that $\P^{\mathsf{SI},10}_{\Delta,\alpha,q}$ stochastically decreases as $\Delta \uparrow E(\L) \cup E(\L^*)$ and hence converges to a probability measure $\P^{\mathsf{SI},10}_{\alpha,q}$. Similarly, $\P^{\mathsf{SI},01}_{\Delta,\alpha,q}$ converges to a measure $\P^{\mathsf{SI},01}_{\alpha,q}$.
We note that this convergence implies that these measures are even-translation-invariant in the following sense.
We call a translation \emph{even} if it preserves $E(\L)$ (and hence also $E(\L^*)$), and \emph{odd} otherwise.
For example, $x \mapsto x+(1,0)$ is an odd translation.
Then it is straightforward to check that $\P^{\mathsf{SI},01}_{\alpha,q}$ and $\P^{\mathsf{SI},10}_{\alpha,q}$ are preserved by even translations. In fact, for any odd translation $T$, we have that
\[ T * \P^{\mathsf{SI},10}_{\alpha,q} = \P^{\mathsf{SI},01}_{\alpha,q} \qquad\text{and}\qquad T * \P^{\mathsf{SI},01}_{\alpha,q} = \P^{\mathsf{SI},10}_{\alpha,q}  .\]
In particular, when the two extremal measures are equal, the common measure is translation-invariant.
We remark that these two measures are Gibbs measures in the usual DLR sense (this can be shown by an adaptation of the arguments used for the usual random-cluster model; see Theorem 4.31 and 4.33 in \cite{grimmett2006random}), though we do not use this fact here.

For $\eta \in \Omega^{\mathsf{SI}}$, let
\[ X^\eta = (\1_{\{\hat\eta_x=(1,1)\}})_{x \in \Z^2} \]
denote the site percolation configuration on $\Z^2$ consisting of open crosses in $\eta$. Let $p_c$ be the critical value for Bernoulli site percolation on $\Z^2$.

\begin{thm}\label{thm:uniqueness}
Let $q>0$ and $\alpha > \frac{p_c}{1-p_c} \cdot \max\{q+1,2\}$. Then
\begin{enumerate}
\item
There exists a translation-invariant probability measure $\P^{\mathsf{SI}}_{\alpha,q}$ on $\Omega^{\mathsf{SI}}$ such that $\P^{\mathsf{SI},\tau}_{\Delta(\Lambda),\alpha,q}$ converges to $\P^{\mathsf{SI}}_{\alpha,q}$ as $\Lambda$ increases to $(\Z^2)^*$ along diamond domains for any $\tau \in \Omega^{\mathsf{SI}}$.

\item If $\eta$ is sampled from $\P^{\mathsf{SI}}_{\alpha,q}$, then $X^\eta$ stochastically dominates a supercritical Bernoulli site percolation on~$\Z^2$. In particular, almost surely, both $\eta^0$ and $\eta^1$ contain a unique infinite cluster.
\end{enumerate}
\end{thm}

We leave open the question of uniqueness of the Gibbs measure and existence of an infinite cluster for general $\alpha$; see \cref{quest:uniqueness,quest:infinite-cluster}.

The proof of the theorem is based on a disagreement percolation argument similar to the one introduced by van den Berg and Maes~\cite{van1994disagreement}.
However, as our model is not defined through a nearest-neighbor interaction, its Gibbs states are not Markov random fields. This presents some difficulty in the argument. A similar issue exists in the random-cluster model, though there, the monotonicity and the facts that both closed and open circuits allow for the use of a domain Markov property are very helpful. In the superimposed model, though we have some monotonicity, the lack of a similar domain Markov property for open circuits in $\eta^0$ (because information can pass through from inside to outside through open crosses) makes the above difficulty persist. However, for circuits of open crosses there is a certain domain Markov property, which we establish below.

We start with the following simple topological claim, the proof of which is straightforward to check and is left to the reader (see also \cref{fig:circuit}).
\begin{lemma}\label{L:top}
Take a path $P$ in $\Z^2$ and let $\xi \in \{0,1\}^{E(\L) \cup E(\L^*)}$ be the configuration in which all the crosses corresponding to vertices in $P$ are open. Let $f,g \in \L$ or $f,g \in \L^*$ be two faces of $\Z^2$, each of which is incident to some vertex in $P$. Then $f$ with $g$ are connected by an open path in $\xi$.
\end{lemma}

We say $x$ and $y$ in $\Z^2$ are $*$-adjacent if either they are neighbors in $\Z^2$ or they belong to diagonally opposite corners of a face in $\Z^2$. For a subset $U$ of $\Z^2$, we write $\partial^* U$ for the external vertex boundary of $U$ in the $*$-adjacency. That is, $\partial^* U$ is the set of vertices in $\Z^2 \setminus U$ which are $*$-adjacent to some vertex in $U$. We define $\partial^* U$ analogously for $U \subset (\Z^2)^*$.

\begin{figure}
\centering
\includegraphics[width = 0.52\textwidth]{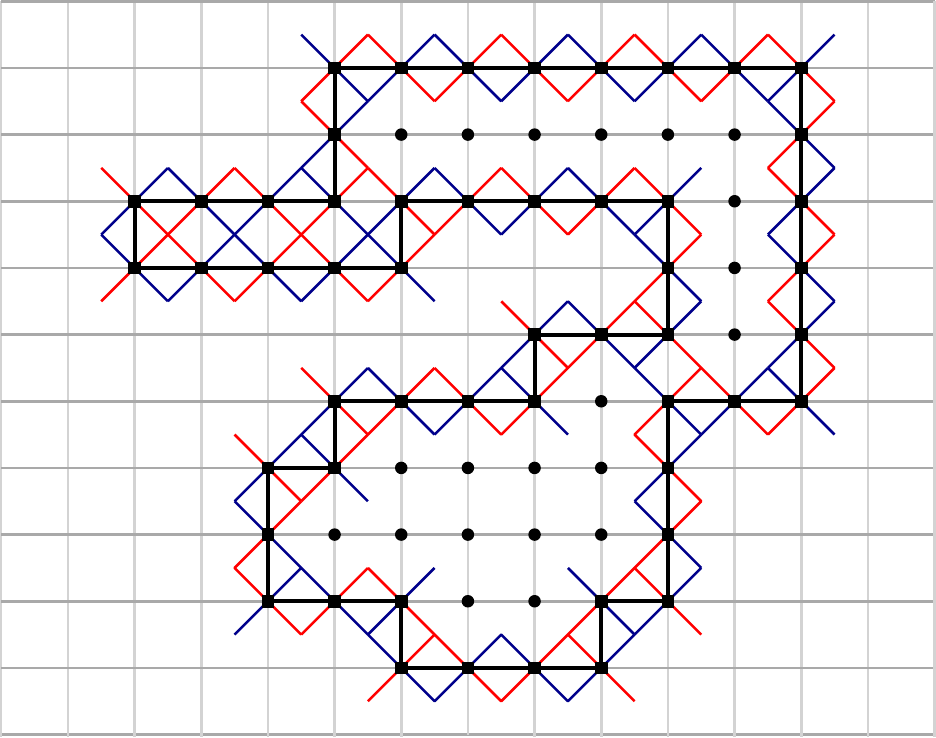}
\caption{Proof of \cref{lem:dmp}. The solid discs denote $\cX$ and the black path is the circuit $\cC$. The primal and dual edges formed by the open crosses in the vertices of $\cC$ are drawn in blue and red, respectively. Note that $\cX$ might not be $*$-connected in $\Z^2$.}
\label{fig:circuit}
\end{figure}

\begin{lemma}\label{lem:dmp}
Let $\cC$ be a simple circuit in $\Z^2$ and let $\cX$ be the set of vertices strictly inside it (i.e., not including the vertices of $\cC$).
Let $\Delta$ be the union of the crosses corresponding to elements in $\cX$. Let $\tau \in \Omega^{\mathsf{SI}}$ be such that the vertices of $\cC$ are open in $X^\tau$. Then for all $\alpha>0$ and $q>0$,
 $$\P^{\mathsf{SI},\tau}_{\Delta,\alpha,q} = \P^{\mathsf{SI},11}_{\Delta,\alpha,q}.$$   
\end{lemma}
\begin{proof}
Let $\eta \in \Omega^{\mathsf{SI},\tau}_\Delta$ and let $\xi \in \Omega^{\mathsf{SI},11}_\Delta$ be the configuration which coincides with $\eta$ in $\Delta$.
It suffices to show that $k_\Delta(\eta^0)=k_\Delta(\xi^0)$ and $k_\Delta(\eta^1)=k_\Delta(\xi^1)$.
Since $\xi$ is obtained from $\eta$ by opening some edges, it suffices to show that any two distinct clusters of $\eta$ (both of which contain a vertex incident to $\Delta$) are also distinct clusters of $\xi$. Indeed, $\eta$ has only one primal (and one dual) cluster containing a face incident to $\cC$. This follows from \cref{L:top} and since $\partial^* \cX \subset \cC$ (see \cref{fig:circuit}).
\end{proof}

We also require the following lemma. Recall that, given a diamond domain $\Lambda$, $\partial^\dagger \Lambda$ denotes the simple circuit in $\Z^2$ which surrounds $\Lambda$ (see \cref{fig:height}) and $\hat \Lambda$ denotes the subgraph of $\Z^2$ induced by vertices of $\partial^\dagger \Lambda$ and all the vertices enclosed by $\partial^\dagger \Lambda$.

\begin{lemma}\label{lem:bc-coupling}
Let $\Lambda$ be a diamond domain and set $\Delta=\Delta(\Lambda)$.  Let $\cX$ be the set of internal vertices of $\hat \Lambda$.
Fix $\alpha>0$ and $q>0$.
Let $\rho$ be a Bernoulli percolation on $\cX$ with parameter $p=\frac{\alpha}{\max\{2,q+1\}+\alpha}$. For each $\tau \in \Omega^{\mathsf{SI}}$, let $\eta^\tau$ be a sample from $\P^{\mathsf{SI},\tau}_{\Delta,\alpha,q}$. Then one may couple $\rho$ and $\{\eta^\tau\}_{\tau \in \Omega^{\mathsf{SI}}}$ so that, almost surely,
\begin{enumerate}
 \item[(i)] $X^{\eta^\tau}_x \ge \rho_x$ for all $x \in \cX$ and $\tau \in \Omega^{\mathsf{SI}}$, and
 \item[(ii)] if $\Gamma$ is an open circuit in $\rho$, then all of $\{\eta^\tau\}_{\tau \in \Omega^{\mathsf{SI}}}$ coincide inside $\Gamma$.
\end{enumerate}
\end{lemma}
\begin{proof}
In this proof it is convenient to consider the $\hat\eta$ representation of $\eta$. For any $x \in \Z^2$, denote by $\{e_x, e_x^*\}$ the cross corresponding to $x$.
We first show that
\begin{equation}
\P^{\mathsf{SI},\xi}_{\{e_x, e_x^*\},\alpha,q}(\hat\eta_x = (1,1))  \ge p \qquad \text{for all }\xi \in \Omega^{\mathsf{SI}}\text{ and }x \in \cX . \label{dom} 
\end{equation}
Denote $e_x=\{u,v\}$ and $e_x^*=\{w,z\}$. Since no cross is closed in $\xi$, either $u \leftrightarrow v$ in $\xi \setminus \{e_x\}$ or $w \leftrightarrow z$ in $\xi \setminus \{e_x^*\}$. Indeed, $\xi^1$ contains the dual of $\xi^0$, so this is a standard consequence of planar duality. Thus, using the calculations in~\eqref{eq:cond_superimposed}, the probability that both $e_x$ and $e_x^*$ are open is either $\frac{\alpha}{2+\alpha}$ or $\frac{\alpha}{q+1+\alpha}$, depending on whether both connectivities exist or not.

The proof now proceeds by rather standard exploration arguments (see \cite[Lemma 9]{duminil2017continuity} for a version involving the random cluster model). We explore the sites in $\cX$ one-by-one starting from the boundary until we discover the outermost open circuits of $\rho$. We also want to ensure that properties (i) holds in each step. 
This can be done as follows. Suppose we have revealed a certain set $X$ of vertices. Then in the next step, we choose any unexplored vertex $x$ that is a $*$-neighbor of some explored vertex $y$ having $\rho_y=0$. If such a vertex exists, then we sample $U_x \sim $ Unif $ [0,1]$ independently of everything else and define
\[ \rho_x = \1_{\{U_x \le p\}} \qquad\text{and}\qquad \eta^\tau_x =
\begin{cases}
 (1,1) &\text{if } U_x \le \P^{\mathsf{SI},\eta^\tau}_{X,\alpha,q}(\hat\eta_x = (1,1)) \\
 (0,1) &\text{if } U_x \ge \P^{\mathsf{SI},\eta^\tau}_{X,\alpha,q}(\hat\eta_x \neq (0,1)) \\
 (1,0) &\text{otherwise}
\end{cases} .\]
Once no such vertex exists, we have discovered the outermost circuits of $\rho$. The above computation shows that (i) holds outside these circuits. We now sample the vertices inside these circuits, all at once. Using \cref{lem:dmp} we can ensure that (ii) holds inside these circuits (and that (i) holds as well).
\end{proof}

\begin{proof}[Proof of \cref{thm:uniqueness}]
To prove the first item, it suffices to show that, for any finite set $A \subset E$, the total-variation distance between the marginals of $\P^{\mathsf{SI},\tau}_{\Delta(\Lambda),\alpha,q}$ and $\P^{\mathsf{SI},11}_{\Delta(\Lambda),\alpha,q}$ on $A$ tends to 0.
By \cref{lem:bc-coupling}, this total-variation distance is at most the probability that a vertex of $\Z^2$ incident to $A$ is connected to $\partial ^\dagger \Lambda$ by a $*$-path which is closed in $\rho$, where $\rho$ is a Bernoulli site percolation with parameter $p=\frac{\alpha}{\max\{2,q+1\}+\alpha}$. Note that the choice of $\alpha$ ensures that $p > p_c$, which implies that the above probability tends to 0 as $\Lambda \uparrow (\Z^2)^*$. This is standard and follows from similar arguments for Bernoulli bond percolation. Indeed, we need to show that $1-p_c^*\le p_c$. If not, pick $p' \in (p_c,1-p_c^*)$ and consider Bernoulli site percolation with parameter $p'$. Using \cite[Proposition 2.1]{duminil2017lectures}, we know that the probability of a left-right open crossing of a square tends to 1 as the size of the square becomes large (the proof is written for bond percolation, but the same proof works for site percolation with any type of connectivity). On the other hand, by the same proposition, the probability of a top-bottom closed *-crossing also tends to 1. By duality, both crossings cannot simultaneously occur which leads to a contradiction. We conclude that $1-p_c^*\le p_c$. The fact that the unique limiting measure $\P^{\mathsf{SI}}_{\alpha,q}$ is translation-invariant now follows from the convergence. This completes the proof of the first item.

The second item follows directly from \cref{lem:bc-coupling}.
The almost sure existence of primal and dual infinite clusters is an immediate consequence of this domination and \cref{L:top}. The uniqueness of the primal (and dual) infinite cluster can be seen as a consequence of a standard Burton--Keane type argument or from the fact that supercritical Bernoulli percolation has an open circuit surrounding any two given sites.
\end{proof}

\subsection{Infinite-volume coupling} \label{sec:inf_coupling}
Here we prove an infinite-volume version of \cref{P:ES_coupling}. We now fix $q = 2$ and $\alpha > \frac{3p_c}{1-p_c}$, where we recall that $p_c$ is the critical value for Bernoulli site percolation on $\Z^2$. Let $\P^{\mathsf{SI}}_\alpha$ be the measure from the first item of \cref{thm:uniqueness} and recall from the second item that there is $\P^{\mathsf{SI}}_\alpha$-almost surely a unique infinite cluster in $\eta^0$ and a unique infinite cluster in $\eta^1$. The coupling we now describe will be between this unique Gibbs measure for the superimposed model and a Gibbs measure for the six-vertex spin configuration (which we will show exists via an infinite-volume limit).
Recall the measure $\bP^{ij}_{\Lambda,\alpha}$ from~\eqref{eq:ES_coupling}.

\begin{prop}\label{P:ES_coupling_infinite}
Let $\alpha > \frac{3p_c}{1-p_c}$ and set $c=2+\alpha$. Let $i,j \in \{+,-\}$ and let $\Lambda_n$ be a sequence of diamond domains such that $\Lambda_n \uparrow (\Z^2)^*$.
Then $\bP^{ij}_{\Lambda_n,\alpha}$ converges as $\Lambda_n \uparrow (\Z^2)^*$ to a limiting measure~$\mu$ on $\{+,-\}^{(\Z^2)^*} \times \Omega^{\mathsf{SI}}$.
The first marginal of $\mu$ is a Gibbs measure $\P^{\textsf{spin},ij}_c$ for the spin representation of the six-vertex model, and the second marginal of $\mu$ is the superimposed measure $\P^{\mathsf{SI}}_\alpha$. Moreover, if $(\sigma,\eta)$ is sampled from $\mu$, then
\begin{itemize}
 \item Given $\sigma,$ $\eta$ can be sampled by first putting in the unique compatible edge at each non-saddle point, and then, independently for each saddle point, assigning one of the three values $(1,0),(1,1),(0,1)$ for $\hat\eta$ with probabilities $\frac1{2+\alpha}, \frac\alpha{2+\alpha},\frac1{2+\alpha}$, respectively.
 \item Given $\eta$, $\sigma$ can be sampled by independently choosing a uniform sign for each finite cluster. The unique infinite cluster in $\eta^0$ receives spin $j$ and the unique infinite cluster in $\eta^1$ receives spin $i$.
\end{itemize}
Moreover, the convergence holds in the following stronger sense. Let  $(\sigma_n, \eta_n)  \sim \bP^{ij}_{\Lambda_n,\alpha}$. Then $$(\sigma_n,\eta_n, \{1_{v \leftrightarrow \infty}\}_v,\{1_{x\leftrightarrow y}\}_{x,y}) \to (\sigma, \eta, \{1_{v \leftrightarrow \infty}\}_v, \{1_{x\leftrightarrow y} \}_{x,y})$$
in distribution as $n \to \infty$. Here the connectivities on the left-hand side are in $\eta_n$ (in which case, $v\leftrightarrow \infty$ simply means that $v$ is connected to the boundary of $\Lambda_n$) and on the right-hand side in $\eta$.
\end{prop}

The reader may be wondering whether \cref{P:ES_coupling_infinite} is a straightforward adaptation of the proof of the analogous statement for the Potts and random-cluster model (e.g. from \cite[Theorem~4.91]{grimmett2006random}). However, this is not the case, since the wired-wired boundary condition for the superimposed model (which is the relevant finite-volume boundary condition here) does not induce the largest measure, as the wired boundary conditions do for the random-cluster model. This is the reason we need to prove the stronger statement (the moreover part of \cref{P:ES_coupling_infinite}), from which everything else follows. Indeed, this shows that the infinite cluster in the limit comes only from the boundary clusters and not from large clusters which do not touch the boundary.

\begin{corollary}\label{cor:spin-correlations}
Let $\alpha > \frac{3p_c}{1-p_c}$ and set $c=2+\alpha$, where $p_c$ is the critical value for Bernoulli site percolation on $\Z^2$. For any $i,j \in \{+,-\}$ and adjacent $u \in \L$ and $v \in \L^*$,
\begin{align*}
\P^{\textsf{spin},ij}_c(\sigma_u=j)=\P^{\textsf{spin},ij}_c(\sigma_v=i) &= \tfrac12 \big[1 + \P^{\mathsf{SI}}_\alpha(v \leftrightarrow \infty) \big], \\
\P^{\textsf{spin},ij}_c(\sigma_u \sigma_v = ij) &=
 \tfrac12 \big[1 + \P^{\mathsf{SI}}_\alpha(u \leftrightarrow \infty,~v \leftrightarrow \infty) \big].
\end{align*}
In particular, the four measures $\{\P^{\textsf{spin},ij}_c: i,j \in \{+,-\}\}$ are all different.
\end{corollary}
\begin{proof}
Let $\sigma \sim \P^{\textsf{spin},ij}_c$ and $\eta \sim \P^{\mathsf{SI}}_\alpha$ be coupled as in \cref{P:ES_coupling_infinite}. Then $\P(\sigma_v=i \mid \eta)$ is 1 if $v \leftrightarrow \infty$ and is $\frac12$ otherwise, and similarly for $u$. Since $\P(u \leftrightarrow \infty)=\P(v \leftrightarrow \infty)>0$, this yields the first statement and shows that the four measures are distinct. Finally, $\P(\sigma_u \sigma_v = ij \mid \eta)$ is 1 if $u \leftrightarrow \infty$ and $v \leftrightarrow \infty$ and is $\frac12$ otherwise. This yields the second statement.
\end{proof}

\begin{proof}[Proof of \cref{P:ES_coupling_infinite}]
We drop the subscript $n$ from the notation of $\Lambda_n$ for clarity. We also denote $\Delta := \Delta(\Lambda)$ and $E := E(\L) \cup E(\L^*)$.
Recall from \cref{P:ES_coupling} that the second marginal of $\bP^{ij}_{\Lambda,\alpha}$ is $\P^{\mathsf{SI},11}_{\Delta,\alpha,2}$. By \cref{thm:uniqueness}, this marginal converges to $\P^{\mathsf{SI}}_\alpha$ as $n \to \infty$. By the description of the conditional measures in \cref{P:ES_coupling}, in order to obtain the proposition, it suffices to show that for any finite set $A \subset (\Z^2)^*$, the joint distribution of $(\1_{\{x \leftrightarrow y\}})_{x,y \in A}$ and $(\1_{\{x \leftrightarrow \infty \}})_{x \in A}$ under $\P^{\mathsf{SI},11}_{\Delta,\alpha,2}$ converges to their joint distribution under $\P^{\mathsf{SI}}_\alpha$ as $\Delta \uparrow E$.  Recall that if $\eta$ is sampled from $\P^{\mathsf{SI},11}_{\Delta,\alpha,2}$, the event $\{x \leftrightarrow \infty \}$ is the same as the event that $x$ is connected to a boundary vertex of $\Lambda$.

For $x,y \in A$ and $m \ge 1$, let $\cE_{x,y,m}$ denote the event that $x$ and $y$ are connected in $\eta$ within a ball of radius $m$ around the origin, and let $\cE_{x,\infty,m}$ denote the event that $x$ is connected in $\eta$ to a vertex at distance $m$ from $x$. Observe that, since these are cylinder events, the distribution of $(\1_{\cE_{x,y,m}})_{x \in A, y \in A \cup \{\infty\}}$ under $\P^{\mathsf{SI},11}_{\Delta,\alpha,2}$ converges to its distribution under $\P^{\mathsf{SI}}_\alpha$ as $\Delta \uparrow E$.
Note also that $\cE_{x,y,m} \subset \{x \leftrightarrow y\}$ and $\cE_{x,\infty,m} \supset \{x \leftrightarrow \infty\}$ for any $x,y\in A$. Thus, for any $x,y \in A$,
\begin{align*}
&\1_{\{x \leftrightarrow y\}} \neq \1_{\cE_{x,y,m}} &\implies\qquad &\{x \leftrightarrow y\} \setminus \cE_{x,y,m} \text{ occurs} ,\\
&\1_{\{x \leftrightarrow \infty \}} \neq \1_{\cE_{x,\infty,m}} &\implies\qquad &\cE_{x,\infty,m} \setminus \{x \leftrightarrow \infty\} \text{ occurs} .
\end{align*}
Therefore, since this is true for any $m$, it suffices to show that, for any $x,y \in A$,
\begin{align*}
&\lim_{m \to \infty} \lim_{\Delta \uparrow E} \P^{\mathsf{SI},11}_{\Delta,\alpha,2}(\{x \leftrightarrow y\} \setminus \cE_{x,y,m}) = 0 &\quad\text{and}\qquad &\lim_{m \to \infty} \P^{\mathsf{SI}}_\alpha(\{x \leftrightarrow y\} \setminus \cE_{x,y,m}) = 0 ,\\
&\lim_{m \to \infty} \lim_{\Delta \uparrow E} \P^{\mathsf{SI},11}_{\Delta,\alpha,2}(\cE_{x,\infty,m} \setminus \{x \leftrightarrow \infty\}) = 0 &\quad\text{and}\qquad &\lim_{m \to \infty} \P^{\mathsf{SI}}_\alpha(\cE_{x,\infty,m} \setminus \{x \leftrightarrow \infty\}) = 0 .
\end{align*}
The two right-hand statements are immediate since $\bigcup_m \cE_{x,y,m} = \{x \leftrightarrow y\}$ and $\bigcap_m \cE_{x,\infty,m} = \{x \leftrightarrow \infty\}$.
We now turn to the two left-hand statements. The first easily follows from the fact that $\P^{\mathsf{SI}}_\alpha$-almost surely there is only one infinite cluster in each of $\eta^0$ and $\eta^1$. For the second statement, we use the coupling from \cref{lem:bc-coupling} between a sample $\eta$ from $\P^{\mathsf{SI},11}_{\Delta,\alpha,2}$ and a Bernoulli percolation $\rho$ on $\Z^2$. Suppose that $\eta \in \cE_{x,\infty,m} \setminus \{x \leftrightarrow \infty\}$ and let $\cC$ be the cluster of $x$ in $\eta$. Thus, $\cC$ is finite and has diameter at least $m$. Suppose without loss of generality that $x \in \L$. Using \cite[Proposition 11.2]{Grimmett_percolation}, we see that this $\cC$ must be blocked by a simple circuit $\Gamma$ in $\L^*$, which clearly has diameter at least $m$.
Thus, the edges $E(\Gamma)$ of $\Gamma$ (which lie in $E(\L^*)$) are all open in $\eta^1$, while all their dual edges are closed in $\eta^0$.
By property~(i) of the coupling, we have that $\eta^0_e \ge \rho_{\{e,e^*\}}$ for all $e \in E(\Gamma)$. Therefore, $\eta^0$ stochastically dominates a Bernoulli edge percolation (on $\L$) of parameter $p>p_c$, and hence, $(\eta^0)^*$ is stochastically dominated by a Bernoulli edge percolation $\rho'$ (on $\L^*$) of parameter $1-p<1-p_c$.
We conclude that $\lim_{\Delta \uparrow E} \P^{\mathsf{SI},11}_{\Delta,\alpha,2}(\cE_{x,\infty,m} \setminus \{x \leftrightarrow \infty\})$ is at most the probability that $\rho'$ contains an open circuit which surrounds $x$ and has diameter at least $m$. Since $\rho'$ is subcritical (the critical site percolation parameter is $p_c>\frac12$, whereas the critical edge percolation parameter is $\frac12>1-p$), this probability decays exponentially to 0 as $m \to \infty$.
\end{proof}

We do not know how to extend \cref{P:ES_coupling_infinite} to all $\alpha>0$, but a similar argument yields a modified version in which the infinite clusters also receive random spins can be shown for all $\alpha>0$ (see \cref{P:ES_coupling_infinite2} below for a precise formulation). This is simply because the spin of every cluster, finite or infinite, is uniformly random (assuming uniqueness of infinite clusters) and connectivity to the boundary is irrelevant for the proof in this case (i.e., the event $\cE_{x,\infty,m}$ above is irrelevant). We do not provide a proof as we do not need this in what follows, and leave it to the reader to fill in the details.

\begin{prop}\label{P:ES_coupling_infinite2}
Let $\alpha > 0$ and set $c=2+\alpha$.
Let $(\Lambda_n)_n$ be a sequence of diamond domains increasing to~$(\Z^2)^*$.
Suppose that $\P^{\mathsf{SI},11}_{\Delta(\Lambda_n),\alpha,2}$ converges as $n \to \infty$ to some measure $\P^{\mathsf{SI},11}_\alpha$ with at most one infinite cluster in each sublattice.
Then $\frac14 \sum_{i,j \in \{+,-\}} \bP^{ij}_{\Lambda_n,\alpha}$ converges as $n \to \infty$ to a limiting measure $\mu$ on $\{+,-\}^{(\Z^2)^*} \times \Omega^{\mathsf{SI}}$.
The first marginal of $\mu$ is a Gibbs measure $\P^{\textsf{spin}}_c$ for the spin representation of the six-vertex model, and the second marginal of $\mu$ is the superimposed measure $\P^{\mathsf{SI},11}_\alpha$. Moreover, if $(\sigma,\eta)$ is sampled from $\mu$, then
\begin{itemize}
  \item Given $\sigma,$ $\eta$ can be sampled by first putting in the unique compatible edge at each non-saddle point, and then, independently for each saddle point, assigning one of the three values $(1,0),(1,1),(0,1)$ for $\hat\eta$ with probabilities $\frac1{2+\alpha}, \frac\alpha{2+\alpha},\frac1{2+\alpha}$, respectively.
 \item Given $\eta$, $\sigma$ can be sampled by independently choosing a uniform sign for each cluster (including any infinite clusters).
\end{itemize}
\end{prop}

\subsection{Proof of Theorem~\ref{thm:h_grad}}
\label{sec:h_grad}

Throughout this section, we fix
\[ c>\frac{2+p_c}{1-p_c} \qquad\text{and}\qquad \alpha = c-2 ,\]
where $p_c$ is the critical value for Bernoulli site percolation on $\Z^2$. Recall that \cref{thm:uniqueness} yields a unique superimposed measure $\P^{\mathsf{SI}}_\alpha$, and that it is translation-invariant and $\P^{\mathsf{SI}}_\alpha$-almost surely there are unique primal and dual infinite clusters in $\eta$. Recall $\P^{\textsf{hf},m}_{\Lambda,c}$ from \eqref{eq:hf} and $\P^{\textsf{spin},ij}_c$ from \cref{P:ES_coupling_infinite}.

\begin{prop}\label{prop:h_limit}
Fix $m \in \Z$. Then $\P^{\textsf{hf},m}_{\Lambda,c}$ converges to a limit $\P^{\textsf{hf},m}_c$ as $\Lambda \uparrow (\Z^2)^*$ along diamond domains.
\end{prop}
\begin{proof}
Set $i,j \in \{+,-\}$ so that $j=+$ if and only if $m=0,3$ mod 4, and $i=+$ if and only if $m=0,1$ mod 4.
By \cref{P:ES_coupling_infinite}, a sample $\sigma$ from $\P^{\textsf{spin},ij}_c$ can be obtained by sampling $\eta$ from $\P^{\mathsf{SI}}_\alpha$ and assigning random spins to finite clusters of $\eta$ and spins $i$ and $j$ to the two infinite clusters. In fact, by the `moreover' part of \cref{P:ES_coupling_infinite}, the limit of $\P^{\textsf{hf},m}_{\Lambda,c}$ exists and a sample from it can be obtained by assigning height $m$ and $m+1$ to the two infinite clusters (belonging to lattices of appropriate parity) and then using $\sigma$ to determine the height on all $(\Z^2)^*$ (recall that $\sigma$ lifts to a unique height function up to an additive integer in $4\Z$).
\end{proof}
Recall the definition of the diagonal gradient $\nabla_d h$ from \eqref{eq:grad_d_h}. 

\begin{prop}\label{prop:h_grad_ffiid}
Fix $m \in \Z$ and let $h$ be sampled from $\P^{\textsf{hf},m}_c$. Then $|\nabla_d h|$ is \ffiid.
\end{prop}

\begin{proof}
For a spin configuration $\sigma \in \{+,-\}^{(\Z^2)^*}$, let $\nabla_d \sigma$ be the random field on the non-directed edges of $\L$ and $\L^*$ defined by $(\nabla_d \sigma)_{\{u,v\}} = \1_{\{\sigma_u \neq \sigma_v\}}$ for $\{u,v\} \in E(\L) \cup E(\L^*)$. 

 Note that if $\sigma$ is the spin configuration obtained from a height function $h$, then $2\nabla_d \sigma = |\nabla_d h|$ and $\sigma$ has law $\P^{\textsf{spin},ij}_c$ for suitable $i,j \in \{+,-\}$ (recall that pushing $\P^{\textsf{hf},m}_{\Lambda,c}$ forward via the projection from height functions to spin configurations yields $\P^{\textsf{spin},ij}_{\Lambda,c}$). Thus, to establish the theorem, it suffices to show that $\nabla_d \sigma$ is \ffiid\ when $\sigma$ is sampled from $\P^{\textsf{spin},ij}_c$.

Let $\eta$ be sampled from $\P^{\mathsf{SI}}_\alpha$.
We aim to apply~\cite[Theorem~2.1]{harel2018finitary}, which roughly says that a monotone model whose extremal measures coincide is \ffiid, to obtain that $\eta$ is \ffiid.
Indeed, the superimposed model is monotone (\cref{prop:FKG}) and $\P^{\mathsf{SI},01}_{\alpha,q}=\P^{\mathsf{SI},10}_{\alpha,q}$ (\cref{thm:uniqueness}).
However, our model is monotone with respect to a ``reversed'' partial order, whereas the result in~\cite{harel2018finitary} is stated for the usual pointwise order.
Thus, we may apply this result to $\hat\eta$, viewed as an element of $\{-1,0,1\}^{\Z^2}$, which is indeed monotone with respect to the usual order (recall the discussion in \cref{sec:monotone}), to obtain that $\hat\eta$ is \ffiid. Going back to $\eta$, this yields that $\eta$ is $(\Z^2)_{\text{even}}$-\ffiid\ (note that the $\hat\cdot$ operation, and hence also its inverse, is not translation-equivariant, but rather $(\Z^2)_{\text{even}}$-equivariant; more specifically, $T \hat\eta = - \hat{T \eta}$ for any odd translation $T$). We emphasize here that this argument proves that the pair $(\eta^0,\eta^1)$ is \emph{jointly} $(\Z^2)_{\text{even}}$-\ffiid.

To get the full invariance for $\eta$, we note that the proof of \cite[Theorem~2.1]{harel2018finitary} extends to the setting of a ``reversed'' partial order. Indeed, the coding constructed there uses coupling-from-the-past for a monotone single-site dynamics, and outputs the value of an edge once it identifies that the dynamics started from the two extremal configurations agree on the state of that edge. Since determining the latter does not depend on which of the two extremal configurations is minimal or maximal, we conclude that the coding is translation-equivariant. Thus, $\eta$ is \ffiid. 

Let $\sigma$ be sampled from $\P^{\textsf{spin},ij}_c$ and recall that our goal is to show that $\nabla_d \sigma$ is \ffiid. Let $\phi$ be a finitary coding from an \iid\ process $Y=(Y_e)_{e \in E(\L) \cup E(\L^*)}$ to $\P^{\mathsf{SI}}_\alpha$ and write $\eta = (\eta^0,\eta^1) = \phi(Y)$.
By \cref{P:ES_coupling_infinite}, $\sigma$ is obtained from $\eta$ by assigning spin $i$ and $j$ to the dual and primal infinite clusters, and assigning independent unbiased signs to the finite clusters.
Applying \cref{thm:grad_general} on the graph $\L$ with the percolation process $\eta^0$ and the spin configuration $\sigma|_{\L}$, we see that $(\nabla_d \sigma)|_{E(\L)}$ (which is the same as $\nabla (\sigma|_{\L})$ with the $\nabla$ used in the theorem) is a finitary factor of $(\eta^0,\xi^0)$, where $\xi^0$ is an \iid\ process on $\L$, independent of $\eta^0$. We may further assume that $\xi^0$ is independent of $\eta$. In other words, there exists an $\L$-equivariant function $\varphi$ such that $\varphi(\eta^0,\xi^0)$ equals $(\nabla_d \sigma)|_{E(\L)}$. Since $(\nabla_d \sigma)_{E(\L)}$ and $(\nabla_d \sigma)|_{E(\L^*)}$ have the same distribution (since $\eta$ is translation-invariant), the same function $\varphi$ also serves as a finitary coding for $(\nabla_d \sigma)|_{E(\L^*)}$ so that $\varphi(\eta^1,\xi^1)$ equals $(\nabla_d \sigma)|_{E(\L^*)}$, where $\xi^1$ has the same distribution as $\xi^0$, and is independent of $\eta$ and $\xi^0$. Putting this together, we see that $(\nabla_d \sigma)_e = \varphi(\eta^0,\xi^0)_e$ for all $e\in E(\L)$ and $(\nabla_d \sigma)_e = \varphi(\eta^1,\xi^1)_e$ for all $e \in E(\L^*)$. Recalling that $\eta = (\eta^0,\eta^1) = \phi(Y)$, we conclude that $\nabla_d \sigma$ is \ffiid.
\end{proof}

\begin{remark}\label{rem:c=2}
Consider the case $c=2$. It is shown in~\cite{glazmanpeled2019} that the gradient of $\P^{\textsf{hf},m}_{\Lambda,c}$ converges (that is, the six-vertex measures converge) and the limiting measure does not depend on $m$. It can be shown that this measure is \ffiid\ (unlike the situation for large $c$).
Indeed, using the fact that the random-cluster model with $q=4$ undergoes a continuous phase transition (established recently in \cite{duminil2017continuity}), \cite[Theorem~1.1]{harel2018finitary} implies that the critical random-cluster measure with $q=4$ is \ffiid.
Recall from \cref{rem:bkw} that, in this case, the superimposed model with $q=2$ coincides with the critical random-cluster model with $q=4$.
Thus, the superimposed model has no infinite clusters, and hence \cref{P:ES_coupling_infinite2} shows that the spin representation of the six-vertex model is \ffiid.
\end{remark}

The above establishes the second part of \cref{thm:h_grad} (as the claim about $|\Delta h|$ then follows easily from~\eqref{eq:spin-height-abs-delta}). For the first part of the theorem, we require the following lemma. Let $\Lambda_n$ denote the diamond domain of diameter $2n$ centered around the origin.

\begin{lemma}\label{lem:spin-agree}
Let $i,j,i',j' \in \{+,-\}$.
There exists a coupling between $\sigma \sim \P^{\textsf{spin},ij}_c$ and $\sigma' \sim \P^{\textsf{spin},i'j'}_c$ such that
\[ \P(\sigma = \sigma' \text{ on } \Lambda_n) \ge e^{-an} \qquad\text{for some }a>0\text{ and all }n \ge 1 .\]
\end{lemma}
\begin{proof}
We use the coupling described in \cref{P:ES_coupling_infinite}, where we couple $\sigma$ and $\sigma'$ with the same sample $\eta \sim \P^{\mathsf{SI}}_\alpha$.
In this coupling, $\sigma$ and $\sigma'$ agree on the finite clusters of $\eta$. Thus, it suffices to show that $\P(H) \ge e^{-an}$, where $H$ is the event that no element of $\Lambda_n$ is in an infinite cluster of $\eta$. To see this, let $\Gamma$ be a simple circuit in $\L$ surrounding $\Lambda_n$ and contained in $\Lambda_{n+2}$, and let $\Gamma'$ be a simple circuit in $\L^*$ surrounding $\Lambda_{n+4}$ and contained in $\Lambda_{n+6}$. Closing all edges in $\Gamma$ and $\Gamma'$ and opening all their dual edges clearly ensures that $H$ holds. By finite energy (of $\hat\eta$), this operation has an exponential cost in $|\Gamma|+|\Gamma'|=O(n)$.
\end{proof}

Recall the definition of the Laplacian $\Delta h$ from~\eqref{eq:eq:lap_h}.

\begin{prop}\label{prop:h_grad_no_ffiid}
Fix $m \in \Z$ and let $h$ be sampled from $\P^{\textsf{hf},m}_c$. Then $\Delta h$ is not $(\Z^2)_{\text{even}}$-\ffiid.
\end{prop}
\begin{proof}
Assume without loss of generality that $m$ is even.
Let $h'$ be sampled from $\P^{\textsf{hf},m-1}_c$.
Since $h$ and $h'$ can be coupled so that $h'=2m-h$ almost surely, it suffices to show that it cannot be that $\Delta h$ and $\Delta h'$ are both $(\Z^2)_{\text{even}}$-\ffiid.
Thus, we assume towards a contradiction that both $\Delta h$ and $\Delta h'$ are $(\Z^2)_{\text{even}}$-\ffiid.

Let $\sigma$ and $\sigma'$ denote the spin configurations obtained from $h$ and $h'$, respectively.
Recall that for $v \in \L$, $\Delta h_v = \frac14 \sum_{u \sim v} (\1_{\{\sigma_u=\sigma_v\}} - \1_{\{\sigma_u \neq \sigma_v\}})$ from~\eqref{eq:spin-height-delta}. Since $m$ is even, $\sigma$ has distribution $\P^{\textsf{spin},++}_c$ or $\P^{\textsf{spin},--}_c$, so that \cref{cor:spin-correlations} implies that $$\E(\Delta h_v)= \P^{\mathsf{SI}}_\alpha(u \leftrightarrow \infty,~v \leftrightarrow \infty).$$
Since $\P^{\mathsf{SI}}_\alpha$-almost surely $\eta^0$ and $\eta^1$ contain infinite clusters and since $\hat\eta$ has finite energy, we see that $\E(\Delta h_v)>0$. A similar calculation shows that $\E(\Delta h'_v)<0$.

Define
\[ Z_n := \frac{1}{|\Lambda_n \cap \L|} \sum_{v \in \Lambda_n \cap \L} \Delta h_v \qquad\text{and}\qquad Z'_n := \frac{1}{|\Lambda_n \cap \L|} \sum_{v \in \Lambda_n \cap \L} \Delta h'_v .\]
Note that $Z_n$ and $Z'_n$ are measurable with respect to $\sigma|_{\Lambda_{n+1}}$ and $\sigma'|_{\Lambda_{n+1}}$, respectively.
As both $\Delta h$ and $\Delta h'$ are $(\Z^2)_{\text{even}}$-\ffiid\ (by our assumption), we can deduce that the convergence in the ergodic theorem occurs at an exponential rate~\cite{bosco2010exponential} (the statement there is for the entire group of translations, but the argument does not require this). Hence,
\[ \P(Z_n \le 0) + \P(Z'_n \ge 0) \le Be^{-bn^2} \qquad\text{for some }B,b>0\text{ and for all }n \ge 1 .\]
In particular, under any coupling of $\sigma$ and $\sigma'$, for all $n \ge 1$,
\[ \P(\sigma|_{\Lambda_{n+1}} = \sigma'|_{\Lambda_{n+1}}) \le \P(Z_n=Z'_n) \le \P(Z_n \le 0) + \P(Z'_n \ge 0) \le Be^{-bn^2} .\]
However, by \cref{lem:spin-agree}, there exists a coupling such that $\P(\sigma|_{\Lambda_{n+1}} = \sigma'|_{\Lambda_{n+1}}) \ge e^{-an}$ for some $a>0$ and all $n \ge 1$.
We have thus reached a contradiction.
\end{proof}

\begin{proof}[Proof of \cref{thm:h_grad}]
\cref{prop:h_limit} shows that $\P^{\textsf{hf},m}_{\Lambda,c}$ converges to an infinite-volume limit $\P^{\textsf{hf},m}_c$ as $\Lambda$ increases to $(\Z^2)^*$ along diamond domains. \cref{prop:h_grad_ffiid} shows that if $h$ is sampled from $\P^{\textsf{hf},m}_c$, then $|\nabla_d h|$ is \ffiid. It then easily follows from~\eqref{eq:spin-height-abs-delta} that $|\Delta h|$ is also \ffiid.

\cref{prop:h_grad_no_ffiid} shows that $\Delta h$ is not $(\Z^2)_{\text{even}}$-\ffiid. By~\eqref{eq:spin-height-nabla} and~\eqref{eq:spin-height-delta}, $\Delta h$ is a finitary factor (with a bounded coding radius) of $\nabla h$, so that it follows that $\nabla h$ is not $(\Z^2)_{\text{even}}$-\ffiid. Similarly, $h$ is also not $(\Z^2)_{\text{even}}$-\ffiid.
Finally, to show that $\nabla_d h$ is not $(\Z^2)_{\text{even}}$-\ffiid, it suffices to show that $\nabla h$ is a finitary factor of $\nabla_d h$. Indeed, this easily follows from the following observations: First, if we know the $\nabla h $ along one edge, then the gradient at any other edge is determined by summing the diagonal gradients along any two diagonal paths connecting their endpoints. Second, we observe that if $(\nabla_d h)_e \neq 0$ for some diagonal edge $e$, then the gradient along the edges whose endpoints are in $\{u,v,u^*,v^*\}$ (where we write $e=\{u,v\}$ and $e^*=\{u^*,v^*\}$) are determined. Finally, with probability 1, there must be an edge with non-zero diagonal gradient (since otherwise all the height would be in $\{m,m+1\}$).
\end{proof}

\section{Open questions}
\label{sec:open}
In this section, we discuss some open questions and future directions of research. We split this section into two subsections, the first dealing with questions solely about the superimposed model, and the second outlining some questions related to finitary codings of gradient models. 
\subsection{Superimposed model}
The first two questions are related to possible extensions of \cref{thm:uniqueness}.
\begin{question}\label{quest:uniqueness}
Is it true that $\P^{\mathsf{SI},10}_{\alpha,q}=\P^{\mathsf{SI},01}_{\alpha,q}$ for all $\alpha>0$ and $q > 1$? 
\end{question}

\begin{question}\label{quest:infinite-cluster}
Fix $\alpha>0$ and $q>1$. Does there exists an infinite cluster under unfavorable boundary conditions, i.e., is $\P^{\mathsf{SI},01}_{\alpha,q}(\text{exists an infinite cluster in $\eta^0$})>0$ (or equivalently equal to 1)?
\end{question}

We also raise the possibility of some monotonicity in the parameter $\alpha$.

\begin{question}\label{quest:monotone}
Fix $q>1$, a finite set $\Delta \subset E(\L) \cup E(\L^*)$ and a boundary condition $\tau \in \Omega^{\mathsf{SI}}$. Is the marginal of $\P^{\mathsf{SI},\tau}_{\Delta,\alpha,q}$ on $\eta^0$ stochastically increasing in $\alpha$?
\end{question}

\subsection{Finitary codings for gradient models}

Let us discuss some questions regarding general models on $\Z^d$ ($d \ge 2$).
The results in this article may be seen as instances of the following type of situation.
Suppose we are given a model with multiple Gibbs states. Let $X$ denote a sample from one of the Gibbs states. Suppose $f$ is a local map for which the law of $f(X)$ is unique (in the sense that all Gibbs states yield the same law). We would like to ask whether this implies that $f(X)$ is \ffiid. In fact, we could also ask this under the weaker assumption that the law of $f(X)$ is unique among all periodic maximal-pressure Gibbs states (this is the situation for the low-temperature Potts model in more than two dimensions, where there are Dobrushin states which induce different gradient measures).

We do not know whether to expect such a general statement to be true.
Indeed, there are more basic questions which are still open.
For simplicity, let us restrict ourselves to models with nearest-neighbor interactions, where the Gibbs measures are Markov random fields.
For models with finite-energy (this assumption may be weakened), uniqueness of the Gibbs measure is a necessary condition for being \ffiid. This immediately raises the question of whether this is also a sufficient condition (perhaps under some mild technical conditions).
As far as we know, even the question of whether this is sufficient for being a factor of an \iid\ process (without the finitary property) is still open (see~\cite[Question~3]{van1999existence}). This makes it somewhat difficult to formulate a very concrete and general (yet tractable) question regarding gradient models, but we nevertheless try to indicate some possible questions of interest in this direction.

For monotone (FKG) models with finite-energy, it is known that uniqueness is sufficient for being \ffiid. We therefore raise the following general question.

\begin{question}
Consider a nearest-neighbor monotone model with finite-energy. Suppose $f$ is a local map such that the law of $f(X)$ is unique among all Gibbs states. Is $f(X)$ \ffiid?
\end{question}

Coming back to the meta-question raised above, we would also be interested in any specific models in which there are multiple Gibbs states and $f(X)$ is \ffiid\ for some interesting function $f$.

One particular instance concerns the critical planar (ferromagnetic) Potts model with $q \ge 5$ states. In this case, it is known that the phase transition is discontinuous~\cite{duminil2016discontinuity,ray2019short} and that there are $q+1$ translation-invariant Gibbs states at criticality; $q$ ordered states arising from constant boundary conditions and one disordered state arising from free boundary conditions.

\begin{question}
Fix $q \ge 5$ and let $\sigma$ be sampled from one of the $q$ constant boundary condition Gibbs states for the critical $q$-state Potts model on $\Z^2$.
Does there exist a non-trivial function $f$ such that $f(\sigma)$ is \ffiid?
\end{question}

Another particular instance concerns the (lattice) Widom--Rowlinson model on $\Z^d$ at high fugacity. In this model, a configuration $\sigma$ consists of spins taking values in $\{-1,0,1\}$ with the hard constraint that neighboring spins cannot have opposite signs (i.e., $\sigma_v\sigma_u \neq -1$ for all adjacent $u$ and $v$). There is a fugacity $\lambda>0$ associated to non-zero spins, so that (in a finite domain) configurations are chosen with probability proportional to $\lambda^{\sum_v |\sigma_v|}$. It is well known~\cite{lebowitz1971phase,burton1995new} that when $\lambda$ is sufficiently large (as a function of $d$), there are two distinct extremal Gibbs measures, which we call here the plus and minus Gibbs measures, related to each other by a global flip of the spins. In this case, these measures are not \ffiid\ (see \cite[Section~1.1.4]{spinka2018finitarymrf}), and the question of whether a gradient of theirs is arises. The gradient we consider here is simply the pointwise absolute value $|\sigma|$ (a more informative gradient would be the random field $(|\sigma_v - \sigma_u|)_{u,v \in \Z^d}$ defined for \emph{all} pairs of vertices, not just nearest neighbors; to keep things simple, we do not consider this here).
We remark that while the Widom--Rowlinson model has a graphical representation similar to that of the beach model~\cite{haggstrom1996random,haggstrom1998random}, its associated random-cluster model is not monotone, and therefore cannot be shown to be \ffiid\ in the same way that the beach-random-cluster was shown to be (namely, using the general result in~\cite{harel2018finitary}). Still, we expect the following to be true.

\begin{question}
Fix $d \ge 2$ and $\lambda>0$. Suppose that the random-cluster model associated to the Widom--Rowlinson model at fugacity $\lambda$ has a unique Gibbs measure. Let $\sigma$ be sampled from the plus Gibbs measure for the Widom--Rowlinson model. Is $|\sigma|$ \ffiid?
\end{question}

Let us also mention an instance in which there is no known graphical representation. Consider the anti-ferromagnetic $q$-state Potts model on $\Z^d$ at low temperature. When $q>4d$, there is a unique Gibbs measure at any temperature, and it is known that this measure is \ffiid~\cite{spinka2018finitarymrf}. On the other hand, it has recently been shown that there are multiple Gibbs measures when $d$ is sufficiently high and the temperature sufficiently low as functions of $q$~\cite{peled2020long,peled2018rigidity}. Moreover, the extreme periodic Gibbs states (of maximal entropy when the temperature is zero) are related to one another by a permutation of the $q$ states and, when $q$ is odd, perhaps also a translation of the lattice by a unit vector in one coordinate. In order not to deal with issues arising from translations (see below), we focus here on even~$q$. There are various possibilities for the choice of a function $f$ for which the law of $f(\sigma)$ is unique among all such periodic Gibbs measures, and we suggest here one such choice. 

\begin{question}
Fix $q \ge 4$ even and let $d$ be sufficiently large and $\beta$ sufficiently large (perhaps infinity).
Let $\sigma$ be sampled from a periodic Gibbs state (of maximal entropy if $\beta=\infty$) for the anti-ferromagnetic $q$-state Potts model at inverse temperature $\beta$. Let $f(\sigma) \in \{0,1\}^{\Z^d}$ be defined by $f(\sigma)_v := \1_{\{ \sigma_v = \sigma_{v+e_1+e_2}\}}$, where $e_1$ and $e_2$ are the first two standard basis vectors in~$\Z^d$. Is $f(\sigma)$ \ffiid? Is $f'(\sigma)$ \ffiid\ for some other non-trivial function $f'$?
\end{question}

We end by addressing the issue of Gibbs states which are related to one another by a translation. The results in this paper concern models in which the relevant Gibbs states are obtained from one another by an ``in-place'' transformation which does not require a translation (note that in the six-vertex model, while the two Gibbs states of arrow configurations are related to each other by a translation, they may also be related to each other by flipping the arrows. In terms of the height function this can be seen as negating the heights, and in terms of the spin representation this can be seen as flipping the spins on one sublattice). When a translation is necessary in order to relate the Gibbs states, it is not even clear how to define a function $f$ for which the law of $f(X)$ is unique.

Let us consider the hard-core model as an example.
In the hard-core model, a random independent set $\sigma \subset \Z^d$ is chosen with probability proportional (in a finite domain) to $\lambda^{|\sigma|}$, where $\lambda>0$ is a parameter called the fugacity. It is well known that when $\lambda$ is sufficiently large (as a function of $d$), there are two distinct extremal Gibbs measures, which we call here the even and odd Gibbs measures, related to one another by a translation by a unit vector. As these measures are not translation-invariant, only $(\Z^d)_\text{even}$-invariant, where $(\Z^d)_\text{even}$ is the group of translations which preserve the two sublattices, it seems natural to ask about $(\Z^d)_\text{even}$-factors. These measures are not themselves $(\Z^d)_\text{even}$-\ffiid\ (this follows from a simple modification of~\cite[Theorem~1.3]{spinka2018finitarymrf} and is similar to the argument given in the proof of \cref{prop:h_grad_no_ffiid}), and we raise the question of whether some gradient of theirs is.

\begin{question}
Fix $d \ge 2$ and let $\lambda$ be sufficiently large. Let $\sigma$ be sampled from the even Gibbs measure for the hard-core model at fugacity $\lambda$. Does there exist a non-trivial function~$f$ such that $f(\sigma)$ is $(\Z^d)_\text{even}$-\ffiid?
\end{question}

\bibliographystyle{abbrv}
\bibliography{library}
\end{document}